\documentclass[12pt,reqno]{amsart}
\usepackage{hyperref}
\usepackage{amsfonts,latexsym,amssymb} 
\newtheorem{lemma}{\bf{Lemma} }[section]
\newtheorem{proposition}{\bf{Proposition}}[section]
\newtheorem{theorem}{\bf{Theorem}}[section]
\newtheorem{remark}{\sc{Remark} }[section]
\newtheorem{definition}{\sc{Definition} }[section]

\usepackage{graphicx,float,multicol}

\def\build#1^#2{\mathrel{\mathop{\kern 0pt#1}\limits^{#2}}}

\usepackage{siunitx}

\begin{document}

 \title[Thermoelectrochemical problem with power-type boundary effects]{Weak solutions for multiquasilinear elliptic-parabolic systems. Application to  thermoelectrochemical problems}

\author{Luisa Consiglieri}
\address{Luisa Consiglieri, Independent Researcher Professor, European Union}
\urladdr{\href{http://sites.google.com/site/luisaconsiglieri}{http://sites.google.com/site/luisaconsiglieri}}

\begin{abstract}
This paper investigates the existence of weak solutions of biquasilinear 
boundary value problem for a coupled elliptic-parabolic system of divergence form with discontinuous
leading coefficients. 
The mathematical framework addressed in the article  considers
 the presence of an additional nonlinearity  in the model 
 which reflects the radiative thermal  boundary effects in some applications of interest.
The results are obtained via the Rothe-Galerkin method.  Only weak assumptions are made on the data and the boundary
conditions are allowed to be on a general form.
The major contribution of the current paper is the explicit expressions
for the constants appeared in the quantitative estimates that are derived. 
These detailed and
explicit estimates may be useful for the study on nonlinear problems that appear in the real world
applications. In particular, they clarify the smallness conditions.
In conclusion, we illustrate how the above results may be applied to the 
thermoelectrochemical phenomena in an electrolysis cell. This problem has several
applications as for instance to optimize  the cell design and operating conditions.
\end{abstract}

\keywords{Rothe-Galerkin method, radiative thermal boundary effects, thermoelectrochemical system}
\subjclass[2010]{35R05, 35J62, 35K59,  78A57, 80A20, 35Q79}
\maketitle

\section{Introduction}

The main gap between theory and practice is the unrealistic
  assumptions that are  usually made by the mathematicians  because they work in their theoretical results.
Among them, they are the constant coefficients of the time derivative term in parabolic equations,
 or its independence on the space variable (commonly the density).
In the real world applications, there are three terms that destroy the regularity of the solutions.
The first quasilinear term  classically
stands for the spatial gradient of the solution, second one stands for the time derivative, 
and the third one appears from the power-type boundary condition.
This power-type boundary condition represents the radiative heat transfer existent on a part of boundary.
We mention to \cite{pouso} for  the transient radiative heat transfer equations in the one-dimensional slab.

Quantitative estimates take the characteristics of the coefficients into account,
but usually include constants that hide some intrinsic characteristics of the domain.
We seek for the complete explicitness of the constants that are involved on
the quantitative estimates, and their effectiveness. We emphasize that  their sharpness remains as an open
problem.
 The main purpose is the analysis of a weak formulation of the corresponding
boundary- and initial-value elliptic-parabolic problem. To that aim, we approximate the problem
via implicit time discretization, by the classical Rothe method.

We point out that, in addition to the fact that Galerkin and  Rothe methods are
 convenient tools for the theoretical analysis of elliptic and evolution  problems \cite{alp,doudu,kacur,roub}, 
it is of particular interest from the numerical point of view \cite{gudi,kacurmah,juha}.
Different versions of the primal discontinuous Galerkin methods to treat the coupling
of flow and transport and the coupling of transport and reaction have recently
gained popularity because they are easier to implement than most traditional finite
element methods, from a computer science point of view (see \cite{sun2005} and the references therein).
Lipschitz continuity property is commonly assumed as a data character, which simplifies the Rothe method
\cite{azab,plus}.
 
The paper \cite{bumi} deals with modeling of quasilinear thermoelectric phenomena, including the Peltier and
Seebeck effects. 
In \cite{cheung}, 
 the spatial distribution of the variables such as the 
electrolyte temperature, which is subject to local cell conditions, 
is studied. 
To optimize cell operations is the aim for the long term sustainability of the aluminum smelting industry.

The mathematical modeling of electrochemical devices such as Lithium-ion battery system \cite{fuller,kupper}
has gaining  of interest in the literature \cite{meth,north,wu-xu}. 
Here, no internal interfaces are considered in the model, which amounts to
neglecting possible material heterogeneities as done in \cite{epjp,jfpta,lap2017}.
These works deal with weak solutions related to thermoelectrochemical devices with radiative effects
in a part of the boundary,
involving the cross effects. 
A particular feature is the mixture of some kind of (nonlinear)  Neumann and Robin
boundary conditions.
 Also, quantitative estimates are stated for the norm (steady-state in \cite{epjp} and unsteady-state in \cite{jfpta})
under appropriate assumptions on the data,
where the constants are given explicitly.
Within this state of mind, we close this paper by applying
the theoretical coupled elliptic-parabolic system to the thermoelectrochemical phenomena.

The structure of the paper is as follows. We begin by introducing the functional framework,
the data under consideration and  the main theorem in Section \ref{begin}.
The main ingredient of the proof is the Rothe method presented in Section \ref{tdt}. 
Section \ref{sfptgalk} deals to the existence proof of the corresponding elliptic problem.
The idea of the proof is based on classical Galerkin approximation argument (Subsection \ref{sgalk}).
In Section  \ref{pass}, we derive a priori estimates for the
approximate problem, getting compactness properties that allow the existence  proof of  the main theorem via
the passage to the limit  as the time-step vanishes.
As a consequence of the main theoretical result,
the existence of a weak solution to a thermoelectrochemical problem  is stated in Section \ref{sappl}.

\section{Introduction}
\label{begin}

 Let $[0, T] \subset {\mathbb R}$ be the time interval with $ T >0
$  being an arbitrary (but preassigned) time. 
Let $\Omega$ be a bounded domain  (that is, connected open set) in $\mathbb{R}^n$ ($n\geq 2$).
Its boundary $\partial\Omega$ is constituted by three pairwise disjoint
 open $(n-1)$-dimensional sets, namely the electrodes surface
$\Gamma$, the wall surface $\Gamma_\mathrm{w}$, and the remaining outer
surface $\Gamma_\mathrm{o}$,
such that $\partial\Omega=\overline{\Gamma}\cup \overline\Gamma_\mathrm{w}
\cup\overline\Gamma_\mathrm{o}$. Observe that the electrodes surface
$\Gamma$ consists of the anode $\Gamma_\mathrm{a}$ and the cathode $\Gamma_\mathrm{c}$.
Figure \ref{cell} displays two schematic geometrical representations of the
domain $\Omega$ and of its boundary $\partial\Omega$
  in order to identify the various subsets into which the boundary is decomposed
and, as a consequence, to better understand the physical significance
of the enforced boundary conditions.
Hence further, we set $Q_T=\Omega\times ]0,T[$ and $\Sigma_T=\partial\Omega\times ]0,T[$.
\begin{figure}
\begin{multicols}{2}
\centering 
 \includegraphics[width=0.6\textwidth]{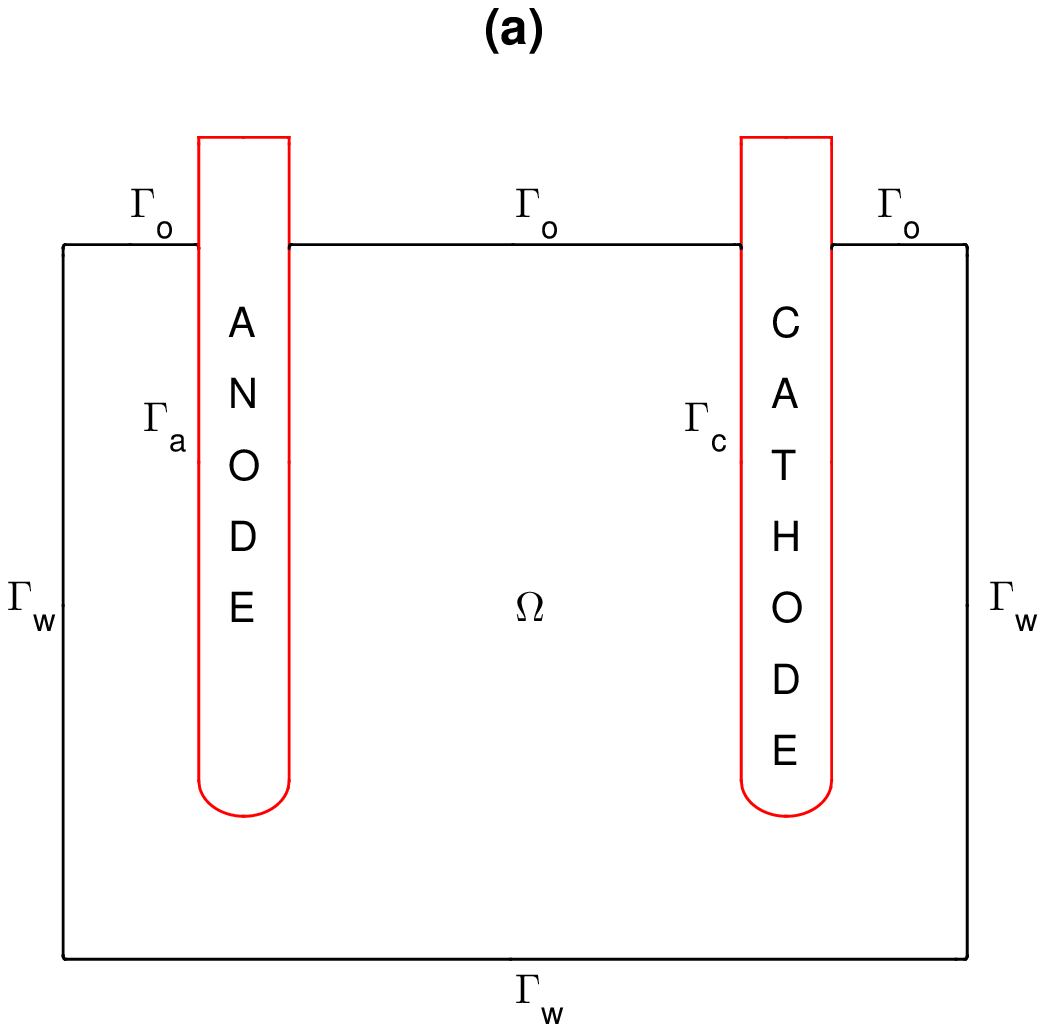} \\
 \includegraphics[width=0.6\textwidth]{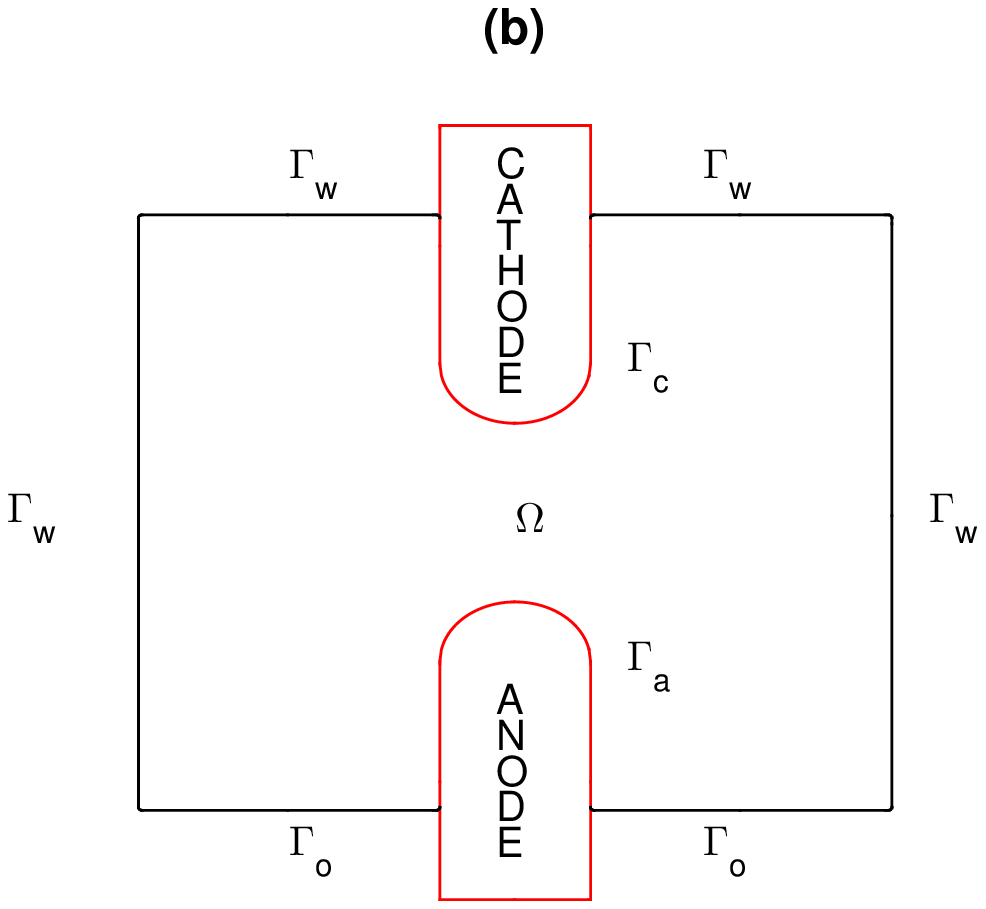}
\end{multicols}
\caption{Schematic 2D representation of two  cells  of one compartment (not in scale). (a) an electrolytic cell.
(b) TEC device design: heating bottom plate and two electrodes symmetrically placed  \cite{nature}.}
\label{cell}     
\end{figure}

We are interested in
the following boundary value problem  in the sense of distributions. Find the functions
 $(\mathbf{u},\phi):Q_T\rightarrow\mathbb{R}^{ \mathrm{I}+2}$,
with $\mathrm{I}$ being an integer number, that solve
\begin{eqnarray}\label{strongj}
\mathsf{B}(u_{\mathrm{I}+1} )\partial_t \mathbf{u} 
-\nabla\cdot \left(\mathsf{A} ( \mathbf{u} )\nabla \mathbf{u} \right)
=\nabla\cdot( \mathbf{F} ( \mathbf{u} )\nabla\phi) ; && \\
-\nabla\cdot(\sigma(  \mathbf{u} )\nabla\phi )=\nabla\cdot( \mathbf{G} ( \mathbf{u} )\nabla\mathbf{u} )
 && \mbox{ in }Q_T,\label{strong1}
\end{eqnarray}
with the following meaning of notation, for $j=1,\cdots ,\mathrm{I}+1$,
\begin{eqnarray*}
\nabla\cdot \left(\mathsf{A} \nabla \mathbf{u}  \right)
&=& \sum_{k=1}^n\partial_k \left(\sum_{l=1}^{\mathrm{I}+1} a_{j,l}\partial_k u_l \right); \\
\nabla\cdot( F_j \nabla\phi)
&=&\sum_{k=1}^n\partial_k (F_j\partial_k\phi) ; \\
\nabla\cdot( \mathbf{G} \nabla\mathbf{u} ) &=& \sum_{k=1}^n \partial_k \left(
\sum_{l=1}^{\mathrm{I}+1} G_l( \mathbf{u} ) \partial_k  u_l \right).
\end{eqnarray*}
Here $\mathsf{A}$ and  $\mathsf{B}$ are $(\mathrm{I}+1)^2$-matrices
such that
\begin{description}
\item[(A)]  the leading matrix $\mathsf{A}$ is supposed to be 
 uniformly elliptic, of quadratic-growth, and with real-valued  $L^\infty$ components;
\item[(B)]  $\mathsf{B}$ is the diagonal matrix with non-zero components
\[
b_{j,j}=\left\{\begin{array}{ll}
1 &\mbox{if } 1\leq j\leq \mathrm{I}\\
b&\mbox{if } j=\mathrm{I}+1 .
\end{array}\right. \]
\end{description}
Only $(\mathrm{I}+1)$ parabolic equation is in fact known as the doubly
nonlinear elliptic-parabolic equation which has been investigated
by several authors when Dirichlet conditions are taken into account on the boundary
 (we refer for example to the works \cite{alt,plus}
and the references cited therein for some details).

The Kirchoff transformation could be applied to the $(\mathrm{I}+1)$ parabolic equation
in order to be useful in the time discretization because
\begin{equation}\label{kirch}
b(u)\partial_t u= \partial_t \left( \int^ub(z)\mathrm{dz} \right),
\end{equation}
 although 
it is not truly useful as change variable
 because the function $b$ depends on the space variable and
$\nabla  \left( \int^ub(r)\mathrm{dr} \right)$ may be ill-defined.

The  boundary conditions are in the concise form
\begin{eqnarray}\label{equal}
 \left(\mathsf{A} ( \mathbf{u}  )\nabla \mathbf{u} + \mathbf{F} ( \mathbf{u} )\nabla\phi \right)
\cdot \mathbf{n} + \mathbf{b} (u_{\mathrm{I}+1})^\top \mathbf{u} =\mathbf{h} ; &&\\
\left( \sigma(  \mathbf{u}  )\nabla \phi +  \mathbf{G} ( \mathbf{u} )\nabla\mathbf{u}\right)\cdot{\bf n} =g\chi_{\Gamma}
&&\mbox{ on }\Sigma_T ,\label{wbc}
\end{eqnarray}
with $\bf n$ denoting the outward unit normal to the boundary $\partial\Omega$, and
\begin{eqnarray*}
b_{j}=\left\{\begin{array}{ll} 0 &\mbox{if } 1\leq j\leq \mathrm{I}\\
\gamma &\mbox{if } j=\mathrm{I}+1 .
\end{array}\right. 
\end{eqnarray*}
Here, the  boundary coefficient $\gamma$ stands for the Robin-type boundary effects
 on $\Gamma$,
 and for the power-type boundary effects on $\Gamma_{\mathrm{w}}$.
The functions $\mathbf h$ and $g$ stand for the boundary sources.

Finally, let the initial condition be
\begin{equation}
\mathbf{u}(\cdot,0)= \mathbf{u}^0\mbox{ in }\Omega .\label{ci}
\end{equation} 

In the framework of Sobolev and Lebesgue functional spaces, we 
use the following spaces of test functions:
\begin{eqnarray*}
V(\Omega) &=&\{ v\in H^{1}(\Omega):\ \int_{\Omega} v \mathrm{dx}=0\}; \\
V(\partial\Omega) &=&\{ v\in H^{1}(\Omega):\ \int_{\partial\Omega} v \mathrm{ds}=0\}; \\
V_{\ell}(\Omega) &=&\{v\in H^{1}(\Omega):\ v|_{\Gamma_\mathrm{w} }\in L^{\ell}(\Gamma_\mathrm{w})\};\\
V_{\ell}(Q_T) &=&\{v\in L^2(0,T;H^{1}(\Omega)):\ 
 v|_{\Gamma_\mathrm{w} \times ]0,T[ }\in L^{\ell}(\Gamma_\mathrm{w} \times ]0,T[)\},
\end{eqnarray*}
with their usual norms,  $\ell>1$.
Hereafter, we use the notation "ds" for the surface element in
the integrals on the boundary  as well as any subpart of the  boundary $\partial\Omega$.
 Notice that $
V_{\ell}(\Omega)\equiv H^{1}(\Omega)$ if $\ell<2_*$, where $2_*$ is the critical trace continuity constant, \textit{i.e.}
$2_*=2(n-1)/(n-2)$ if $n>2$
and $2_*>1$ is arbitrary if $n=2$.

The problem (\ref{strongj})-(\ref{strong1}) is in fact a system of $\mathrm{I}+2$ partial differential equations
and it may be decomposed in one  system of $\mathrm{I}$ parabolic equations, one parabolic equation with a
quasilinear time derivative, 
and one third elliptic equation.
\begin{definition} \label{dwt}
We say that a function $(\mathbf{u},\phi)$ is a weak solution to the problem
 (\ref{strongj})-(\ref{strong1}) and (\ref{equal})-(\ref{ci}),
  if it satisfies (\ref{ci}) and the variational formulation, with $u= u_{\mathrm{I}+1}$,
\begin{eqnarray}
\int^T_0\langle \partial_t u_i ,v_i \rangle \mathrm{dt}+\sum_{j=1}^{\mathrm{I}+1} 
\int_{Q_T} a_{i,j}( \mathbf{u})\nabla u_j \cdot\nabla v_i\mathrm{dxdt}=\nonumber \\
 = -
\int_{Q_T} F_i( \mathbf{u} )\nabla\phi \cdot\nabla v_i\mathrm{dxdt} 
+\int_{\Sigma_T}  h_i v_i \mathrm{dsdt}, \quad i= 1,\cdots, \mathrm{I};\label{wvfi} \\
\int^T_0\langle b(u)\partial_t  u , v\rangle \mathrm{dt}+\sum_{j=1}^{\mathrm{I}+1} 
\int_{Q_T} a_{\mathrm{I}+1,j}( \mathbf{u})\nabla u_j \cdot\nabla v\mathrm{dxdt}
+\int_{\Sigma_T} \gamma(u) uv \mathrm{dsdt} =\nonumber \\
 = -
\int_{Q_T} F_{\mathrm{I}+1}( \mathbf{u} )\nabla\phi \cdot\nabla v\mathrm{dxdt} 
+\int_{\Sigma_T}  h_{\mathrm{I}+1}v \mathrm{dsdt}; \qquad \label{wvfi1} \\
\int_\Omega\sigma ( \mathbf{u} )\nabla\phi\cdot\nabla w\mathrm{dx}
 = - \sum_{j=1}^{\mathrm{I}+1} \int_\Omega G_j( \mathbf{u} )
\nabla u_j \cdot \nabla w \mathrm{dx}+\int_\Gamma g w \mathrm{ds} , \ \mbox{a.e. in }  ]0,T[, \label{wvfphi}
\end{eqnarray}
for all $v_i\in L^2(0,T;V(\Omega) )$,  $v\in V_\ell(Q_T)$, and $w\in V(\partial\Omega) $.
\end{definition}
 The symbol $\langle\cdot,\cdot\rangle$  denotes
the duality pairing $\langle\cdot,\cdot\rangle_{X'\times X }$, with $X$ being a Banach space.
The notation $X'$ denotes the dual space of  $X$, and $X'$ is equipped with
the usual induced norm $\|f\|_{X'}=\sup \lbrace  \langle f,u\rangle, \ u\in X: \|u\|_X \leq 1 \rbrace$.

The set of hypothesis is as follows.
\begin{description}

\item[(H1)] The vector-valued functions $\mathbf{F}$ and $\mathbf{G}$,
from $\Omega\times\mathbb{R}^{\mathrm{I}+1}$ into $\mathbb{R}^{\mathrm{I}+1}$,  are assumed to be Carath\'eodory, 
{\em i.e.} measurable with respect to $x\in\Omega$ and
  continuous with respect to other variables, such that verify
\begin{eqnarray}
\exists F_j^\#>0 :&& |F_j (x,\mathbf{e} )|\leq F_j^\# ; \label{hfj}\\
\exists G_j^\#>0 :&& |G_j (x,\mathbf{e} )|\leq G_j^\#, \label{hgj}
\end{eqnarray}
for all $j=1,\cdots, \mathrm{I}+ 1$, for a.e. $x\in\Omega$, and for all $\mathbf{e}\in\mathbb{R}^{\mathrm{I}+1}$.
 
\item[(H2)]  The coefficient  $b$  is assumed to be a Carath\'eodory function 
from $\Omega\times\mathbb{R}$ into $\mathbb{R}$. Moreover,
there exist $b_\#,b^\#>0$ such that
\begin{equation} 
b_\#\leq b(x,e)\leq b^\#,\label{bmm}
\end{equation} 
for a.e. $x\in\Omega$, and for all $e\in\mathbb{R}$.

\item[(H3)]
The leading coefficient $\mathsf{A}$ has its components  $a_{i,j}:\Omega\times\mathbb{R}^{\mathrm{I}+1}
\rightarrow\mathbb{R}$
being Carath\'eodory functions. Moreover, they satisfy
\begin{eqnarray}\label{aaii}
 (a_{i})_\#: = \min_{(x,\mathbf{e})\in \Omega\times\mathbb{R}^{\mathrm{I}+1} }
a_{i,i} (x, \mathbf{e})>0; \\
\label{aij}
\exists a_{i,j}^\#>0:\quad
|a_{i,j}(\cdot,\mathbf{e})|\leq a_{i,j}^\#,\quad \mbox{a.e. in } \Omega,\ \forall
\mathbf{e}\in\mathbb{R}^{\mathrm{I}+1} ,
\end{eqnarray}
for all $i,j\in\{1,\cdots,\mathrm{I}+1\}$.

\item[(H4)]
The leading coefficient  $\sigma$  is assumed to be a Carath\'eodory function
from $\Omega\times\mathbb{R}^{\mathrm{I}+1}$ into $\mathbb{R}$. Moreover, there exist $\sigma_\#,\sigma^\#>0$  such that
\begin{equation}
 \sigma_\#\leq \sigma(x,\mathbf{e})\leq \sigma^\#, \label{smm}
\end{equation}
for a.e. $x\in\Omega$, and for all $\mathbf{e}\in\mathbb{R}^{\mathrm{I}+1}$.

\item[(H5)] The boundary coefficient $\gamma$  is assumed to be a Carath\'eodory function 
from $\partial\Omega\times\mathbb{R}$ into $\mathbb{R}$. 
 Moreover,
there exist $\gamma_\#,\gamma^\#>0$ and $\gamma_1\geq 0$ such that 
\begin{equation}\label{gamm}
\gamma_\#|e|^{\ell -2}\leq \gamma(\cdot,e) \leq \gamma^\# |e|^{\ell-2} +\gamma_1,
\end{equation}
 a.e. in $\partial\Omega$, and  for all $e\in\mathbb{R}$,
 where the exponent $\ell\geq 2$ stands for the Robin-type boundary condition ($\ell=2$)
 on $\Gamma$,
 and for the power-type boundary condition ($\ell>2$) on $\Gamma_{\mathrm{w}}$.
\end{description}

\begin{remark}
The boundary condition (\ref{gamm}) may be generalized for a function
 $\gamma_1 : \partial\Omega \rightarrow\mathbb{R}$ belonging to $L^{\ell/(\ell-2)} (\partial\Omega)$
for $\ell\geq 2$. Indeed, Theorem \ref{texist} remains valid if (\ref{gamm}) is replaced by
\begin{eqnarray*}
| \gamma(\cdot,e)|
\leq \gamma_1  &\mbox{ a.e. on }\Gamma ;\\
\gamma_\#|e|^{\ell -2}\leq \gamma(\cdot,e) \leq \gamma^\# |e|^{\ell-2} +\gamma_1 &\mbox{ a.e. on }\Gamma_\mathrm{w},
\end{eqnarray*}
 for all $e\in\mathbb{R}$, which infer in Section \ref{sgalk} that the Brouwer fixed point theorem is applied for a different
 $r>0$ taking Definition \ref{def2} into account.
\end{remark}

Hereafter, we will use the Kirchoff transformation (\ref{kirch}) to the time derivative term, \textit{i.e.} 
the characterization $\partial_t B(u)$, denoting
 by $B$ the operator defined by
\begin{equation}\label{defb}
v\in L^2(Q_T)\mapsto B(v)=\int_0^v b(\cdot,z)\mathrm{dz}.
\end{equation}

Let us state the existence results.
\begin{theorem}\label{texist}
Suppose that the assumptions
(H1)-(H5), $h_i\in L^{2}(\Sigma_T)$, $i=1,\cdots,\mathrm{I}$,  $h_{\mathrm{I}+1}\in L^{\ell/(\ell-1)}(\Sigma_T)$,
and  $ g\in L^{2}(\Gamma)$ be fulfilled. Under the smallness conditions, for  $i\in\{1,\cdots,\mathrm{I}+1\}$,
 \begin{eqnarray}
(a_i)_\# & >&\frac{1}{2}\left(\sum_{l=1\atop l\not=i}^{\mathrm{I}+1} (
a_{i,l}^\#+ a_{l,i}^\#)+F_i^\#+G_i^\#\right),\label{saii} \\
\label{small2}
 \sigma_\# & > & \frac{1}{2}\sum _{ j=1}^{\mathrm{I}+1} \left( F_j^\#+ G_j^\#\right),
 \end{eqnarray}
there exists  at least one  weak solution 
$(\mathbf{u}, \phi) \in [ L^\infty(0,T;L^2(\Omega)) ]^{\mathrm{I}+1} \times L^2(0,T;V(\partial\Omega) )$ 
in accordance to Definition \ref{dwt}, with $v \in  L^{\ell}(0,T;V_\ell(\Omega))$, such that
\begin{eqnarray*}
u_i-u^0_i\in L^2(0,T;V(\Omega) ) 
\quad \mbox{and} \quad \partial_t u_i\in  L^2(0,T; (V(\Omega))' ); \\
u\in  V_\ell(Q_T) \quad \mbox{and} \quad b(u)\partial_t u\in  L^{\ell'}(0,T;(V_\ell(\Omega))'),
\end{eqnarray*}
for $i=1,\cdots,\mathrm{I}$.
In particular, $B(u)\in L^\infty(0,T;L^1(\Omega) )$.
\end{theorem} 

Here, we consider the Banach spaces that are of direct application 
for the thermoelectrochemical problem under study. Clearly, Theorem \ref{texist}
 remains valid for any closed subspace $V$ such that  $H^1_0(\Omega)\hookrightarrow V\hookrightarrow  H^1(\Omega)$
is considered instead of  $V(\Omega)$ or $V(\partial\Omega)$
if the Poincar\'e inequality is verified.

\begin{remark}\label{rtd}
In (\ref{wvfi})-(\ref{wvfi1}),
the meaning of the time derivative should be understood as in  the following weak sense \cite{alt}:
\begin{eqnarray}\label{rtdi}
\int^T_0\langle \partial_t u_i ,v_i \rangle \mathrm{dt} &=& 
-\int^T_0 \int_\Omega u_i \partial_t  v_i\mathrm{dx} \mathrm{dt}-
\int_\Omega u_i^0   v_i (0)\mathrm{dx} ; \\ \label{rtd1}
\int^T_0\langle b(u)\partial_t  u , v\rangle \mathrm{dt} &=& -\int^T_0\int_\Omega B(u) 
\partial_t  v\mathrm{dx} \mathrm{dt} - \int_\Omega B(u^0) v(0)\mathrm{dx} ,
\end{eqnarray}
for every test functions $v_i\in  L^{2}(0,T;V(\Omega))\cap
W^{1,1}(0,T;L^\infty(\Omega) )$, for  $i\in\{1,\cdots,\mathrm{I}\}$, and $v\in  L^{\ell}(0,T;V_\ell(\Omega))\cap
W^{1,1}(0,T;L^\infty(\Omega) )$ such that $v_i(T)=v(T)=0$ a.e. in $\Omega$.
\end{remark}

\section{Time discretization technique}
\label{tdt}

We adopt 
the weak solvability of $\mathrm{I}+1$ time dependent partial differential
equation with a nonlinear Neumann boundary condition as 
investigated in \cite{alt,kacur99}, while the $j$ parabolic equations ($j=1,\cdots,\mathrm{I}$) are studied via
 the classical  time discretization technique \cite{kacur}.
We introduce a recurrent system of boundary value problems to be successively  solved for $m = 1,\cdots,M\in\mathbb N$, 
starting from the initial function (\ref{ci}).

We decompose the time interval $I=[0,T]$ into $M$
subintervals $I_{m,M}$ of size
$\tau$ (commonly called time step) such that $M=T/\tau \in\mathbb N$,  \textit{i.e.}
$I_{m,M}=[(m-1)T/M,mT/M]$ for  $m\in\{1,\cdot\cdot\cdot,M\}$.
We set $t_{m,M}=mT/M$.

For any time integrable function $h:\Sigma_T\rightarrow \mathbb{R}$,
we set
\begin{equation}\label{barh}
\bar{h}^m=\frac{1}{\tau}\int_{(m-1)\tau}^{m\tau} h(\cdot,z)\mathrm{dz}.
\end{equation}

Then, the problem (\ref{wvfi})-(\ref{wvfphi}) is approximated by the following
 recurrent sequence of time discretized problems
\begin{eqnarray}
\frac{1}{\tau}   \int_\Omega
 u_i^{m} v_i\mathrm{dx}+
\sum _{j=1}^{\mathrm{I}+1}\int_\Omega
a_{i,j}( \mathbf{u}^m)\nabla u_j^{m}\cdot\nabla v_i\mathrm{dx}+
\int_\Omega F_i( \mathbf{u}^m) \nabla\phi^m \cdot\nabla v_i\mathrm{dx} =\nonumber \\
\label{wvfm} = 
\frac{1}{\tau}  \int_\Omega u_i^{m-1} v_i \mathrm{dx}
+\int_{\partial\Omega}  \bar{h}^m_iv_i \mathrm{ds} ,\quad i=1,\cdots,\mathrm{I} ;\qquad \\
\frac{1}{\tau}\int_\Omega B (u^m) v\mathrm{dx}+\sum _{j=1}^{\mathrm{I}+1}\int_\Omega
a_{\mathrm{I}+1,j}
( \mathbf{u}^m)\nabla u_j^{m}\cdot\nabla v\mathrm{dx}+
\int_\Omega F_{\mathrm{I}+1}( \mathbf{u}^m) \nabla\phi^m \cdot\nabla v\mathrm{dx} +\nonumber \\
+\int_{\partial \Omega} \gamma(u^m) u^{m} v \mathrm{ds} 
\label{wvfttm} = 
\frac{1}{\tau}\int_\Omega B (u^{m-1} )v\mathrm{dx}
+\int_{\partial\Omega}  \bar{h}^m_{\mathrm{I}+1}  v  \mathrm{ds} ;\qquad \\
\int_\Omega\sigma ( \mathbf{u}^m)\nabla\phi^m\cdot\nabla w\mathrm{dx}
 +\sum_{j=1}^{\mathrm{I}+1} \int_\Omega G_j( \mathbf{u}^m)
\nabla u_j^{m} \cdot \nabla w \mathrm{dx}=\int_\Gamma g w \mathrm{ds} ,\qquad \label{wvfphim}
\end{eqnarray}
where $\mathbf{u}=(u_1,\cdots,u_{\mathrm{I}},u)$,
for all $v_i\in  V(\Omega)$, $i=1,\cdots, \mathrm{I}$,
$v \in V_{\ell}(\Omega)$ and $w\in V(\partial\Omega)$.
Since $\mathbf{u}^0\in L	^2(\Omega)$ is known, we determine $\mathbf{u}^{1}$ as the unique solution
of Proposition \ref{propm}, and we inductively proceed.

The existence of the above system of elliptic problems  is established in the following proposition.
\begin{proposition}\label{propm}
Let $m\in\{1,\cdot\cdot\cdot,M\}$ be fixed, and $\mathbf{u}^{m-1}$ be given.
Then, there exists a unique solution $(\mathbf{u}^{m},\phi^m)\in 
[V(\Omega)]^{\mathrm{I}} \times  V_\ell(\Omega)\times  V(\partial\Omega)$ to the
variational system (\ref{wvfm})-(\ref{wvfphim}).
\end{proposition}
This existence of solution is proved in Section \ref{sfptgalk} via the Galerkin method (cf. Subsection \ref{sgalk}).

Let us recall the technical result \cite{alt,kacur99}. 
\begin{lemma}\label{lbmm}
Denoting by  \[
\Psi (s) := B(s)s-\int_0^s B(r)\mathrm{dr} = \int_0^s(B(s) - B(r) )\mathrm{dr},
\]
there holds
\begin{equation}\label{bpsi}
\int_\Omega (B(u)-B(v) ) u \mathrm{dx} \geq 
\int_\Omega \Psi(u)\mathrm{dx} - \int_\Omega \Psi(v)\mathrm{dx} .
\end{equation}
In particular, if  the assumption (\ref{bmm}) is fulfilled then there holds
\[
\int_\Omega \Psi(u)\mathrm{dx}\leq 
\int_\Omega B(u) u \mathrm{dx} \leq b^\# \| u\|_{2,\Omega}^2 .
\]
Under the assumption (\ref{bmm}) the operator $B$ verifies
\begin{equation}\label{bb}
(B(u)-B(v),u-v)\geq b_\# \|u-v\|_{2,\Omega}^2 .
\end{equation}
\end{lemma}

In order to  control the time dependence, we begin by recalling the following remarkable lemma \cite[Lemma 1.9]{alt}.
\begin{lemma} \label{lbmm2}
Suppose $u_m$ weakly converge to $u$ in $L^p(0,T;W^{1,p}(\Omega))$, $p>1$, with the estimates
\[
\int_\Omega \Psi(u_m(t))  \mathrm{dx}\leq C\quad \mbox{for } 0<t<T,
\]
and for $z>0$
\begin{equation}\label{cbb}
 \int_0^{T-z}
\int_\Omega (B(u_m(t+z))-B(u_m(t)) ) (u_m(t+z)-u_m(t)) \mathrm{dx} \mathrm{dt} \leq Cz ,
\end{equation}
with $C$ being  positive constants.
Then, $B(u_m) \rightarrow B(u)$ in $L^1(Q_T)$ and
$\Psi (u_m) \rightarrow \Psi(u) $ almost everywhere in  $Q_T$.
\end{lemma}

In the sequel, we will also need  both the discrete Gronwall inequality and the 
Aubin-Lions theorem. Let us recall the following discrete version of the  Gronwall inequality \cite{kacur99}.
  \begin{lemma}[Discrete Gronwall inequality]\label{lgronw}
  Let $\{a_m\}_{m\in\mathbb{N}}$ and  $\{A_m\}_{m\in\mathbb{N}}$ be sequences of nonnegative real numbers such that
$A_m$ is nondecreasing and
\[
a_m\leq A_m +\tau L \sum_{j=1}^{m} a_{j},  \]
for each $m\in\mathbb{N}$ and for some $0< \tau L< 1$.
Then, there holds
 \[
a_m\leq \frac{A_m}{1-\tau L} \exp [(m-1)\tau ]
. \]
  \end{lemma}

  Let us recall the following version of the Aubin-Lions theorem for piecewise constant functions \cite{dreher}.
  \begin{theorem}[Aubin-Lions]\label{tal}
Let $X$, $B$, and $Y$ be Banach spaces such that the embeddings $X\hookrightarrow\hookrightarrow B \hookrightarrow Y$ hold,
and let $T>0$ and $1\leq p<\infty$.
Let $\{u_M\}_{M\in\mathbb{N}}$ be a sequence of functions, which are constant on each time subinterval $](k-1)\tau, k\tau ]$
  with uniform time step $\tau=T/M$,
  satisfying
  \[
  \tau^{-1}\|u_M-u_{M-1}\|_{L^1(\tau,T;Y)} +  \|u_M\|_{L^p(0,T;X)}\leq C_0,\quad \forall\tau>0,
  \]
  where $C_0$ is a positive constant independent on $\tau$. Then, 
  there exists a subsequence of $\{u_M\}_{M\in\mathbb{N}}$ strongly  converging in $L^p(0,T;B)$.
  \end{theorem}

\section{Proof of Proposition \ref{propm}}
\label{sfptgalk}

Let $m\in\{1,\cdot\cdot\cdot,M\}$ be fixed, and $\mathbf{u}^{m-1}$ be given.
Set $\mathbf{f}=\mathbf{u}^{m-1}$, and $\mathbf{g}$ be such that
\begin{equation}\label{defg}
g_j=\left\{\begin{array}{ll} 
\bar{h}^m_j &\mbox{if } 1\leq j\leq \mathrm{I} +1 \\
g \chi_{\Gamma} &\mbox{if } j=\mathrm{I}+2.
\end{array}\right.
\end{equation}
Set the $(\mathrm{I}+2)^2$-matrix
\begin{equation}\label{defl}
\mathsf{L} ( \mathbf{u})=\left[
\begin{array}{cc}
\mathsf{A}( \mathbf{u})& \mathbf{F}( \mathbf{u} ) \\
\mathbf{G}^\top ( \mathbf{u} )
& \sigma ( \mathbf{u} )
\end{array}
\right] .
\end{equation}
Using the assumptions (\ref{hfj}), (\ref{hgj}) and (\ref{aaii})-(\ref{smm}) we find
\begin{equation}\label{coer}
\sum_{j,l=1}^{\mathrm{I} +2 }\sum_{\iota=1}^n
\left( L_{j,l} ( \mathbf{u})\xi_{l,\iota}\right)\xi_{j,\iota}\geq
\sum_{j=1}^{\mathrm{I} +2 }\sum_{\iota=1}^n
(L_{j})_\# |\xi_{j,\iota}|^2 ,
\end{equation}
 where, for $j=1,\cdots, \mathrm{I}+1$,
 \begin{eqnarray*}
(L_j)_\# &=& (a_j)_\#- \frac{1}{2}\left(\sum_{l=1\atop l\not=j}^{\mathrm{I}+1} (
a_{l,j}^\#+ a_{j,l}^\#)+F_j^\#+G_j^\#\right) ;\\
(L_{\mathrm{I}+2})_\# &=&
 \sigma_\#- \frac{1}{2}\sum _{ j=1}^{\mathrm{I}+1} \left( F_j^\#+ G_j^\#\right).
 \end{eqnarray*}

\begin{remark}\label{rdet}
Although  the positive-definiteness implies invertibility,
there are invertible  matrices that are not positive definite.
The existence of  the inverse matrix $\mathsf{L}^{-1}$ 
may be consequence of det$(\mathsf{L})\not=0$.
An alternative sufficient condition  is that 
rank$\left( \mathsf{L}\right)=\mathrm{I}+2$.
\end{remark}

\begin{definition}\label{def2}
 We call by  $K_2(P_2+1)$ the constant that verifies
\begin{equation}\label{k2p2}
\|v\|_{2,\Gamma}\leq K_2 \left(
\|v\|_{2,\Omega}+ \| \nabla v\|_{2,\Omega}\right)
\leq K_2(P_2+1) \|\nabla v\|_{2,\Omega},
\quad \forall v\in H^1(\Omega).
\end{equation}
Here, $K_{2}$ stands to  the continuity constant of the trace embedding
$H^{1}(\Omega)\hookrightarrow L^2(\Gamma)$, and
 $P_2$ stands to the Poincar\'e constant correspondent to the space
exponent $2$.
\end{definition}

\subsection{Galerkin approximation technique}
\label{sgalk}

The Banach space $\mathbf{V}:=[V(\Omega)]^\mathrm{I}
\times  V_\ell(\Omega)\times  V(\partial\Omega)$ admits linearly independent functions 
$\mathbf{w} ^\nu$, $\nu=1,\cdots, N$,
 such that the finite-dimensional
subspace $\mathbf{V}_N =\mathrm{span}\{\mathbf{w}^1, \cdots , \mathbf{w}^N \}$
 is dense in $\mathbf{V}$,  for every $N\in {\mathbb N}$.

Introduce  the continuous function
$P: \mathbb{M}_{ (\mathrm{ I}+2)\times N }\rightarrow  \mathbb{M}_{ (\mathrm{ I}+2)\times N }$ 
that maps
$\left[\lambda_{ j,\nu }\right]$ into $\left[ \beta_{ j,\nu }\right]$, 
defined by for each $\nu=1,\cdots, N$
\begin{eqnarray*}
\beta_{j,\nu}&=&\frac{1}{\tau}\int_\Omega  U^N_jw_j^\nu\mathrm{dx}+
\sum_{l=1}^{\mathrm{I}+2}
\int_\Omega \left( L_{j,l}( \mathbf{U}^N)\nabla U_l^N\right)\cdot\nabla  w_j^\nu \mathrm{dx}\\
 && -\frac{1}{\tau}\int_\Omega  f_jw_j^\nu\mathrm{dx}
-\int_{\partial\Omega} g_j w_j^\nu\mathrm{ds}, \quad \forall j=1,\cdots,\mathrm{I};\\
\beta_{j,\nu}&=&\frac{1}{\tau }\int_\Omega b( U^N_{ \mathrm{ I}+1 })  U^N_jw_j^\nu\mathrm{dx}+
\sum_{l=1}^{\mathrm{I}+2}
\int_\Omega\left( L_{j,l}( \mathbf{U}^N)\nabla U_l^N\right)\cdot\nabla  w_j^\nu  \mathrm{dx} + \\
  && +\int_{\partial\Omega}\gamma ( U^N_{ \mathrm{ I}+1 })  U^N_jw_j^\nu\mathrm{ds}
 -\frac{1}{\tau}\int_\Omega b( U^N_{ \mathrm{ I}+1 })  f_jw_j^\nu\mathrm{dx}
-\int_{\partial\Omega} g_j w_j^\nu\mathrm{ds}, \quad  j=\mathrm{I}+1; \\
\beta_{j,\nu}&=& \sum_{l=1}^{\mathrm{I}+2}
\int_\Omega \left( L_{j,l}( \mathbf{U}^N)\nabla U_l^N\right)\cdot\nabla  w_j^\nu \mathrm{dx}
 -\int_{\partial\Omega} g_j w_j^\nu\mathrm{ds}, \quad j=\mathrm{I}+2 ,
\end{eqnarray*}
with the  function $\mathbf{U}^{N}\in  \mathbf{V}_N$ being in the form
\[ U_j^{N}(x)=\sum_{\nu=1}^N \lambda_{j,\nu}^N w_j^\nu(x),\quad j=1,\cdots,\mathrm{I}+2.\]

In order to apply the Brouwer fixed point theorem \cite{ll}, we must prove that
 $P$ satisfies $(P\lambda,\lambda)> 0$
  for all $\lambda\in  \mathbb{M}_{ (\mathrm{ I}+2) \times N}$ such that
$ |\lambda |=\left( \sum_{j=1}^{ \mathrm{ I}+1 }\sum_{\nu=1}^N\lambda_{j,\nu}^2 \right) ^{ 1/2 }=r$,
and $(\beta,\lambda)$ stands for the inner product in $ \mathbb{ M }_{ (\mathrm{ I}+2)\times N }$.
To this aim, we compute
\begin{eqnarray*}
(P\lambda,\lambda)&=&\sum_{j=1}^{\mathrm{I}+2}\sum_{\nu=1}^N\beta_{j,\nu}\lambda_{j,\nu}=\\
&=&\frac{1}{\tau} \sum_{j=1}^{\mathrm{I}+1} \int_\Omega
b_{j,j}( U^N_{ \mathrm{ I}+1 }) | U^N_j|^2\mathrm{dx}+\sum_{j=1}^{\mathrm{I}+2}
\sum_{l=1}^{\mathrm{I}+2}
\int_\Omega \left( L_{j,l}( \mathbf{U}^N)\nabla U_l^N\right)\cdot\nabla  U_j^N \mathrm{dx}\\
&& + \int_{\partial\Omega}\gamma 
(U^N_{ \mathrm{ I}+1 }) | U^N_{ \mathrm{ I}+1 } |^2\mathrm{ds} -\frac{1}{\tau} \sum_{j=1}^{\mathrm{I}+1}
\int_\Omega b_{j,j}(U^N_{ \mathrm{ I}+1 })  f_jU_j^N\mathrm{dx}
- \sum_{j=1}^{\mathrm{I}+2}\int_{\partial\Omega} g_j U_j^N\mathrm{ds}.
\end{eqnarray*}

Applying the assumptions (\ref{bmm}) and (\ref{gamm}),
 the H\"older inequality, and (\ref{k2p2}), we obtain
\begin{eqnarray*}
(P\lambda,\lambda)&\geq &\frac{1}{\tau} \sum_{j=1}^{\mathrm{I}} \left( 
 \| U^N_j\|_{2,\Omega} - \| f_j\|_{2,\Omega} -K_2
 \| g_j\|_{2,\partial\Omega}  \right)\| U^N_j\|_{2,\Omega} +\\
 &&+ \frac{1}{\tau} 
\left( b_\#\|U^N_{ \mathrm{ I}+1 } \|_{2,\Omega}-b^\#\| f_{\mathrm{I}+1}\|_{2,\Omega}
\right)\|U_{\mathrm{I}+1}^N\|_{2,\Omega} +\\ &&
+\sum_{j=1}^{\mathrm{I}} \left(
( L_{j})_\#\|\nabla U_j^N\|_{2,\Omega} -K_2
 \| g_j\|_{2,\partial\Omega}  \right)\|\nabla  U_j^N\|_{2,\Omega} + \\
 && + (L_{\mathrm{ I}+1})_\#\|\nabla U_{ \mathrm{ I}+1 } ^N\|_{2,\Omega}^2 +\\
&& 
+\left(\gamma_\#  \|U^N_{ \mathrm{ I}+1 } \|_{\ell,\partial\Omega}^{\ell-1} 
+\gamma_1  \|U^N_{ \mathrm{ I}+1 } \|_{\ell',\partial\Omega}
-\| g_{ \mathrm{ I}+1 } \|_{\ell ',\partial\Omega}\right)
\| U^N_{ \mathrm{ I}+1 } \|_{\ell,\partial\Omega} + \\
&& + \left(
( L_{ \mathrm{ I}+2 } )_\#\|\nabla U_{ \mathrm{ I}+2 }^N\|_{2,\Omega} -K_2(P_2+1)
 \| g_{ \mathrm{ I}+2 }\|_{2,\partial\Omega}  \right)\|\nabla  U_{ \mathrm{ I}+2 }^N\|_{2,\Omega} .
\end{eqnarray*}
Then,  there exists $r>0$ such that fulfills
$(P\lambda,\lambda)> 0$.
We are in the position of applying the Brouwer fixed point theorem.
Consequently, there exists $\lambda\in\mathbb{M}_{ (\mathrm{ I}+2) \times N}$
 such that $|\lambda|\leq r$
and $P([\lambda_{j,\nu}])=0$, \textit{i.e.} taking the density of $\mathbf{V}_N$ into $\mathbf{V}$,
\begin{eqnarray}
\frac{1}{\tau}  \sum_{j=1}^{\mathrm{I}} \int_\Omega  U^N_jv_j\mathrm{dx}+
\frac{1}{\tau }\int_\Omega
b( U^N_{ \mathrm{ I}+1 })  U^N_{ \mathrm{ I}+1 }v_{ \mathrm{ I}+1 }\mathrm{dx}+\nonumber \\ +
\sum_{j=1}^{\mathrm{I}+2}
\sum_{l=1}^{\mathrm{I}+2}
\int_\Omega \left( L_{j,l}( \mathbf{U}^N)\nabla U_l^N\right)\cdot\nabla  v_j \mathrm{dx}
  +
  \int_{\partial\Omega}\gamma 
(U^N_{ \mathrm{ I}+1 })  U^N_{ \mathrm{ I}+1 } v_{ \mathrm{ I}+1 }\mathrm{ds} =\nonumber \\
 = \frac{1}{\tau} \sum_{j=1}^{\mathrm{I}} \int_\Omega  f_jv_j\mathrm{dx} +\frac{1}{\tau}\int_\Omega
b(U^N_{ \mathrm{ I}+1 })  f_{ \mathrm{ I}+1 } v_{ \mathrm{ I}+1 }\mathrm{dx} +
\sum_{j=1}^{\mathrm{I}+2}\int_{\partial\Omega} g_j v_j\mathrm{ds}  .\label{galerk}
\end{eqnarray}

In order to pass to the limit  in 
 the variational equality (\ref{galerk}) with $N$, when $N$ tends to infinity,
 we can extract a subsequence, still denoted by $ \mathbf{U}^N$,
 convergent to $\mathbf{U}$ weakly in  $\mathbf{V}$ and strongly in $\mathbf{L}^2(\Omega)$
 and in $\mathbf{L}^2(\partial\Omega)$.  In particular, $ \mathbf{U}^N$
 pointwisely converges to $\mathbf{U}$ a.e. in $\Omega$ and on $\partial\Omega$. 
 Applying the Krasnoselski theorem to the Nemytskii operators $b$ and $\mathsf{L}$,
 we have 
 \begin{eqnarray}
 b(U^N_{ \mathrm{ I}+1 }) v\build\longrightarrow^{N\rightarrow\infty} b(U_{ \mathrm{ I}+1 }) v &\mbox{in }& L^2(\Omega);
\\
\sum_{j=1}^{\mathrm{I}+2}L_{j,l}( \mathbf{U}^N)\nabla v_j \build\longrightarrow^{N\rightarrow\infty}
\sum_{j=1}^{\mathrm{I}+2}L_{j,l}( \mathbf{U})\nabla v_j &\mbox{in }& \mathbf{L}^2(\Omega), 
 \end{eqnarray}
for $l=1,\cdots,\mathrm{I}+1$, and for all $v,v_j\in H^1(\Omega)$, 
making use of the Lebesgue dominated convergence theorem with
 the assumptions (\ref{hfj})-(\ref{smm}).
 Similarly, the boundary term $\gamma (U^N_{ \mathrm{ I}+1 })v$ converges  to
 $\gamma (U_{ \mathrm{ I}+1 }) v$ in $L^{\ell '}(\partial\Omega)$, for all
 $v\in L^{\ell '}(\partial\Omega)$, due to (\ref{gamm}).
 Thus, we are in the condition of passing to the limit in 
 the variational equality (\ref{galerk})
as $N$ tends to infinity to conclude that $\mathbf{U}$ is the required limit solution.

\section{Passage to the limit as time goes to zero ($M\rightarrow +\infty$)}
\label{pass}

Set $\mathbf{X}_{\ell}=
 [V(\Omega)]^{\mathrm{I}}\times V_{\ell}(\Omega)$.
Let
$\widetilde  {\mathbf{u}}^M :]0,T[  \rightarrow \mathbf{X}_{\ell}$ and 
$\widetilde  \phi^M :]0,T[  \rightarrow   V(\partial\Omega)$ be  the step functions defined by
\begin{eqnarray}
\widetilde  {\mathbf{u}}^M (t)=\left\{\begin{array}{ll}
\mathbf{u}^{0}&\mbox{ if }t=0\\
 \mathbf{u}^{m}&\mbox{ if } t\in  ]t_{m-1,M},t_{m,M}]
\end{array}\right. \label{defum} \\
\widetilde\phi^M (t)= \phi ^{m}\quad\mbox{ if } t\in ]t_{m-1,M},t_{m,M}], \label{defpm}
\end{eqnarray}
and let $h^M\in L^\infty (0,T;L^1(\partial\Omega))$ be  the (piecewise constant in time) function given by
$h^M(t)=\bar{h}^m$ for $t\in ] (m-1)\tau, m\tau ]$ (cf. (\ref{barh})).

We begin by establishing the estimates and the weak convergences of the  Rothe function
\[ (\widetilde {\mathbf{u}}^M,\widetilde\phi^M) = (\widetilde u_1^M,\cdots, \widetilde u_\mathrm{I}^M,
\widetilde u^M,\widetilde\phi^M)\] obtained from the discretized solution
$ ( {\mathbf{u}}^m,\phi^m)$,  of
variational system (\ref{wvfm})-(\ref{wvfphim}),
 by piecewise constant interpolation with respect to time $t$.
\begin{proposition}\label{propu}
Denoting by 
$\{ (\widetilde {\mathbf{u}}^M,\widetilde\phi^M)\}_{M\in\mathbb N}$  the Rothe sequence,
then the following estimate hold, for $M>T$,
\begin{eqnarray}\label{cotaul}
\max_{1\leq m\leq M} \left(  \sum _{i=1}^{\mathrm{I}}  \| u_i^m\|_{2,\Omega} ^2+ 2\int_\Omega \Psi( u^m )\mathrm{dx}
\right)  +\nonumber \\ +
  \sum _{i=1}^{\mathrm{I}} (L_{i})_\#  \|\nabla\widetilde u_i^M\|_{2,Q_T} ^2 + 
( L_{\mathrm{I}+1})_\# \| \nabla \widetilde u ^M\|_{2,Q_T} ^2 
+ (L_{\mathrm{I} + 2})_\# \| \nabla \widetilde\phi^M\|_{2,Q_T} ^2 + \nonumber \\ 
+ 2 \frac{\gamma_\#}{\ell '} \|  \widetilde u ^M\|_{\ell,\Sigma_T } ^\ell 
\leq  \left( 1+ \frac{ M}{M-T}\exp [T]\right) \mathcal{R} , 
 \end{eqnarray}
 where
 \begin{eqnarray*}
 \mathcal{R} &=& \sum _{i=1}^{\mathrm{I}}  \| u_i^0\|_{2,\Omega} ^2+ 2 b^\#  \| u^0\|_{2,\Omega} ^2 + 
T \frac{K_2^2(P_2+1)^2}{ (L_{\mathrm{I}+2})_\#} \|g\|_{2,\Gamma_\mathrm{N}}^2 + \\  && +
 K_2^2 \sum_{i=1}^{\mathrm{I}}  \left( 1+
\frac{1}{(L_{i})_\#} \right) \|h_i \|_{2,\Sigma_T} ^2 +
\frac{1}{\ell '\gamma_{\#}^{1/(\ell-1)} } \| h \|_{\ell',\Sigma_T }^{\ell '} .
 \end{eqnarray*}
 Moreover, there exists $(\mathbf{u}, \phi)\in [L^2(0,T;V(\Omega) )]^\mathrm{I}\times V_\ell(Q_T) \times
 L^2(0,T;V(\partial\Omega)) $ such that
\begin{eqnarray*}
\widetilde {\mathbf{u}}^M \rightharpoonup \mathbf{u} &\mbox{ in }&  [L^2(0,T;V(\Omega) )]^\mathrm{I}\times V_\ell(Q_T) 
\hookrightarrow \mathbf{L}^2(0,T;\mathbf{X}_{\ell}); \\
\widetilde \phi^M \rightharpoonup \phi &\mbox{ in }& L^2(0,T;V(\partial\Omega)) ,
\end{eqnarray*}
as $M$ tends to infinity  (up to subsequences).
\end{proposition}
\begin{proof}
Choosing $(\mathbf{v},v)=\mathbf{u}^m$ and $w=\phi^m$
as  test functions in (\ref{wvfm})-(\ref{wvfphim}), we sum the obtained relations, and
we successively apply the H\"older inequality, to deduce
\begin{eqnarray}
\frac{1}{\tau}\left( \sum _{i=1}^{\mathrm{I}} \int_\Omega \left( u_i^{m} - u_i^{m-1} \right) u_i^m\mathrm{dx}+
\int_\Omega (B(u^m)-B(u^{m-1}) )u^m\mathrm{dx} \right)+ \nonumber \\ +
 \sum_{j=1}^{\mathrm{I}+1}(L_{j})_\# \|\nabla u_j^{m} \|_{2,\Omega}^2+
(L_{\mathrm{I}+2})_\#\|\nabla\phi ^{m} \|_{2,\Omega}^2 
+\gamma_\# \| u^{m}\|_{\ell,\partial \Omega} ^\ell \leq \nonumber \\ \leq
 \sum_{ i=1}^{\mathrm{I}}\| \bar{h}^m_i \|_{2,\partial\Omega} \| u_i ^m\|_{2,\partial\Omega}+
\| \bar{h}^m_{\mathrm{I}+1}\|_{\ell',\partial\Omega} \| u^m\|_{\ell,\partial\Omega}+
\|g\|_{2,\Gamma}\|\phi^m\|_{2,\Gamma}  \nonumber \\ 
:=  \sum_{ i=1}^{\mathrm{I}}\mathcal{I}_i +\mathcal{I}_{\mathrm{I}+1}  +\mathcal{I}_{\mathrm{I}+2} ,  \label{eqrbum}
\end{eqnarray}
 for all  $m\in \{1,\cdots,M\}$.
We successively apply (\ref{k2p2}), and the Young inequality,  to obtain
\begin{eqnarray*}
\mathcal{I}_i &\leq & \frac{K_2^2}{2}  \left( 1+
\frac{1}{(L_{i})_\#} \right) \| \bar{h}^m_i \|_{2,\partial\Omega} ^2 +\frac{1}{2} \|u_i ^m\|_{2,\Omega}^2
+ \frac{(L_{i})_\#}{2} \| \nabla u_i^m\|_{2,\Omega}^2 ; \\
 \mathcal{I}_{\mathrm{I}+1} &\leq & \frac{1}{\ell '\gamma_{\#}^{1/(\ell-1)} }
\| \bar{h}^m_{\mathrm{I}+1} \|_{\ell',\partial\Omega}^{\ell '} 
+\frac{\gamma_\#}{\ell}\| u^m\|_{\ell,\partial\Omega}^\ell ; 
\\
 \mathcal{I}_{\mathrm{I}+2} &\leq & \frac{K_2^2(P_2+1)^2}{2 (L_{\mathrm{I}+2})_\#}
\|g\|_{2,\Gamma}^2+ \frac{(L_{\mathrm{I}+2})_\#}{2}\|\nabla\phi^m\|_{2,\Omega} ^2.
\end{eqnarray*}
 Making recourse to the elementary identity $2(a-b)a=a^2-b^2+
 (a-b)^2$ for all $a,b\in\mathbb{R}$  to the first term on the left hand side in (\ref{eqrbum}), 
summing over $k=1,\cdots, m$, we obtain
 \begin{eqnarray*}
\sum_{k=1}^m \sum _{i=1}^{\mathrm{I}}  \int_\Omega \left( u_i^{k} - u_i^{k-1} \right)u_i^k\mathrm{dx}
\geq  \frac{1}{2} \sum _{i=1}^{\mathrm{I}} \left(\|u_i^{m} \ |^2_{2,\Omega}  
- \|u_i^{0} \|^2_{2,\Omega} \right) .
\end{eqnarray*}
Next, applying Lemma \ref{lbmm} we deduce for the second  term on the left hand side in (\ref{eqrbum}) 
\[
\int_\Omega (B(u^m)-B(u^{m-1}) )u^m\mathrm{dx} \geq 
\int_\Omega (\Psi(u^m) - \Psi(u^{m-1})  )\mathrm{dx} .\]

Therefore, summing over $k=1,\cdots, m$
into (\ref{eqrbum}),  inserting the above equalities, and multiplying by $2\tau$, we obtain
\begin{eqnarray*}
 \sum _{i=1}^{\mathrm{I}}  \| u_i^m\|_{2,\Omega} ^2+   2 \int_\Omega \Psi(u^m) \mathrm{dx}+ 
 \tau\sum_{k=1}^m \left(\sum_{i=1}^{\mathrm{I}} (L_{i})_\# 
\| \nabla u_i^k\|_{2,\Omega}^2 + \right. \\ \left. +
2 (L_{\mathrm{I}+1})_\# \|\nabla u^{k} \|_{2,\Omega}^2+ 
 (L_{\mathrm{I}+2})_\# \|\nabla\phi ^{k} \|_{2,\Omega}^2 
+\frac{2\gamma_\#}{\ell '} \| u^{k}\|_{\ell,\partial \Omega} ^\ell\right) \leq  \\
\leq  \sum _{i=1}^{\mathrm{I}}  \| u_i^0\|_{2,\Omega} ^2+ 2\int_\Omega \Psi(u^0) \mathrm{dx}
+\tau \sum_{k=1}^m \sum_{i=1}^{\mathrm{I}} \|u_i ^k\|_{2,\Omega}^2 +\\
+\tau\sum_{k=1}^m\left( K_2^2
 \sum_{i=1}^{\mathrm{I}}  \left( 1+
\frac{1}{(L_{i})_\#} \right) 
 \| \bar{h}^k_i \|_{2,\partial\Omega} ^2 
+\frac{2}{\ell '\gamma_{\#}^{1/(\ell-1)} }
\| \bar{h}^k_{\mathrm{I}+1}\|_{\ell',\partial\Omega}^{\ell '}\right) + \\
+\tau m
 \frac{K_2^2(P_2+1)^2}{ (L_{\mathrm{I}+2})_\#}
\|g\|_{2,\Gamma}^2.
\end{eqnarray*}
In particular, the discrete Gronwall inequality (cf. Lemma \ref{lgronw}), with $L=1$ and $\tau=T/M<1$, implies that
\[
 \sum _{i=1}^{\mathrm{I}}  \| u_i^m\|_{2,\Omega} ^2\leq \frac{M\mathcal{R}}{M-T}\exp [T].
 \]
 	Taking the maximum over $m\in \{1,\cdots,M\}$, the estimate (\ref{cotaul}) holds.

Thus we can extract a subsequence, still denoted by $(\widetilde {\mathbf{u}}^M, \widetilde \phi^M )$,
weakly convergent to $(\mathbf{u},\phi)\in  [L^2(0,T;V(\Omega) )]^\mathrm{I}\times V_\ell(Q_T)
\times L^2(0,T;V(\partial\Omega)).$
   \end{proof}

Let us introduce some Rothe functions obtained by piecewise linear interpolation with respect to time $t$.
\begin{definition} \label{defro}
We say that  $\{ (\mathbf{U}^M, B^M) \}_{M\in\mathbb N}$ is the Rothe sequence if
\begin{eqnarray*}
U_i^M(x,t) &=& u_i^{m-1}(x)+ \frac{t-t_{m-1,M} }{ \tau}\left( u_i^m(x)- u_i^{m-1}(x) \right), \quad i=1,\cdots,\mathrm{I}; \\
  B^M(x,t) &=& B(x,u^{m-1}(x))+\frac{t-t_{m-1,M} }{ \tau}\left(B(x,u^{m}(x))-B(x,u^{m-1}(x))\right) ,
\end{eqnarray*}
for all $(x,t)\in\Omega\times  I_{m, M}$, $m\in \{1,\cdots,M\}$.
  \end{definition} 
  
The discrete derivative with respect to the time has the following characterization.
\begin{proposition}\label{zz}
Let 
 $\widetilde{\mathbf{Z}}^M:[0,T[\rightarrow [L^2(\Omega) ]^{\mathrm{I} +1 }$   be defined by
\[
\widetilde{\mathbf{Z}}^M (t)=\left\{\begin{array}{ll}
\mathbf{Z}^{0}&\mbox{ if }t=0\\
\mathbf{Z}^{m}&\mbox{ if } t\in ]t_{m-1,M},t_{m,M}]
\end{array}\right.\mbox{ in }\Omega.
\]
with $\mathbf{Z}^{0}=(u_1^0,\cdots, u_\mathrm{I}^0,B(u^0))$, and
 the discrete derivative with respect to $t$ at the time $t=t_{m,M}$ being such that
\begin{eqnarray}\label{defzj}
Z_i^{m} &:=& \frac{u_i^{m}-u_i^{m-1} }{\tau}, \quad i=1,\cdots,\mathrm{I};
\\
Z_{\mathrm{I} +1 }^{m}&:=& \frac{B(u^{m})-B(u^{m-1}) }{\tau}.\label{defz1}
\end{eqnarray}
Then, there  exists $\mathbf{Z}\in [L^{2}(0,T;(V(\Omega))')]^\mathrm{I}\times L^{\ell'}(0,T;(V_\ell(\Omega))')$ such that
\begin{equation}\label{zm}
\widetilde {\mathbf{Z}}^M \rightharpoonup \mathbf{Z} \quad\mbox{ in }
 [L^{2}(0,T;(V(\Omega))')]^\mathrm{I}\times L^{\ell'}(0,T;(V_\ell(\Omega))').
\end{equation}
\end{proposition}
\begin{proof}
Let  $\{ (\mathbf{U}^M, B^M) \}_{M\in\mathbb N}$ be the Rothe sequence in accordance with Definition \ref{defro}.
For $i=1,\cdots,\mathrm{I}$,  by definition of norm, we have
\[ 
\|\partial_t U_i^M\|_{L^{2}(0,T;(V(\Omega))')} = 
\sup_{v\in L^2(0,T; V(\Omega)) \atop \|v\|\leq 1} \sum_{m=1}^M
\int_{(m-1)\tau}^{m\tau}  \langle Z_i^m,v\rangle  \mathrm{dt} .
\] 
Applying Proposition \ref{propu} to the equality (\ref{wvfm}) being rewritten as
\[
   \int_\Omega  Z_i^{m} v\mathrm{dx} =
\int_{\partial\Omega}  \bar{h}^m_i v \mathrm{ds} - 
\sum _{j=1}^{\mathrm{I}+1}\int_\Omega
a_{i,j}( \mathbf{u}^m)\nabla u_j^{m}\cdot\nabla v \mathrm{dx}-
\int_\Omega F_i(  \mathbf{u}^m) \nabla\phi^m \cdot\nabla v \mathrm{dx} ,
\]
we conclude
\[ 
\|\partial_t U_i^M\|_{L^{2}(0,T;(V(\Omega))')} \leq C,
\] 
with $C>0$ being a constant independent on $M$.
Analogously, applying Proposition \ref{propu} to the equality (\ref{wvfttm}) we find
\[ 
\|\partial_t B^M\|_{L^{\ell'}(0,T;(V_\ell(\Omega))')} =
\sup_{v\in L^\ell(0,T;V_\ell(\Omega)) \atop \|v\|\leq 1} \sum_{m=1}^M
\int_{(m-1)\tau}^{m\tau}  \langle Z_{\mathrm{I} +1 }^m,v\rangle  \mathrm{dt} 
\leq C,
\] 
with $C>0$ being a constant independent on $M$.

Hence, we can extract a subsequence, still denoted by $\widetilde{\mathbf{Z}}^M$, weakly
convergent to $\mathbf{Z}\in [L^{2}(0,T;(V(\Omega))')]^\mathrm{I}\times L^{\ell'}(0,T;(V_\ell(\Omega))').$
   \end{proof}

In the following proposition, we state some strong convergences that
allow, up to a subsequence,  a.e. pointwise convergences. 
\begin{proposition}\label{pbttm}
Let $(\widetilde {\mathbf{u}}^M, \widetilde \phi^M )$ be according to Proposition \ref{propu}.
Under  (\ref{hfj})-(\ref{aij}) and (\ref{gamm}), 
for a subsequence, there hold 
\begin{eqnarray}
\widetilde {\mathbf{u}}^M \rightarrow \mathbf{u} &\mbox{ in }& L^2(Q_T) ; \label{bfum} \\
B(\widetilde   u^M) \rightarrow B(u) &\mbox{ in }& L^1(Q_T) ,\label{bbm}
\end{eqnarray}
 as $M$ tends to infinity. Also, $\widetilde  u^M$ strongly converges to $u$ in  $L^2(\Sigma_T)$.
\end{proposition}
\begin{proof}
To prove (\ref{bfum}), we make recourse to  the discrete version of the Aubin-Lions theorem \ref{tal}.
Thanks to Proposition \ref{propu}, we have
 \[ 
   \|\widetilde {\mathbf{u}}^M\|_{L^2(0,T;\mathbf{X}_\ell(\Omega))} ^2\leq 
  T  \sup_{t\in ]0,T[}
 \sum _{i=1}^{\mathrm{I}}\|\widetilde u_i^M\|_{2,\Omega} ^2+
  \|  \widetilde u ^M\|_{\ell,\Sigma_T }^\ell +\| \nabla \widetilde {\mathbf{u}}^M\|_{2,Q_T} ^2
\leq C,
   \]   
with $C>0$ being a constant independent on $M$.

 For a fixed $t\in ]0,T[$, there exists $m\in\{1, \cdots, M\}$ such that  $t\in]t_{m-1,M},t_{m, M}]$.
For $i=1,\cdots,\mathrm{I}$, 
by applying (\ref{hfj}) and (\ref{aij})  into (\ref{wvfm})-(\ref{wvfttm}),  we deduce 
\begin{eqnarray*}
  \|u_i^{m} -u_i^{m-1}\|_{V'(\Omega)}\leq {\tau} \sup_{v\in V(\Omega): \ \|v\|=1 }\left(
\| \bar{h}^m_i \|_{2,\partial\Omega}\|v \|_{2,\partial\Omega} + \right.  \\ \left. +
\left(\max (a^\#_{ij}) \|\nabla \mathbf{u}^{m}\|_{2,\Omega} + \max (F^\#_j)
\| \nabla\phi^m \|_{2,\Omega} \right)\|\nabla v\|_{2,\Omega} \right)
 .
\end{eqnarray*}
While   for $i=\mathrm{I}+1$, 
by applying (\ref{bmm}), (\ref{bb}), (\ref{hfj}), (\ref{aij}) and  (\ref{gamm}), we deduce
\begin{eqnarray*}
  \|u^{m} -u^{m-1}\|_{V'_\ell(\Omega)}
\leq  \frac{\tau}{b_\#} \sup_{v\in V_\ell (\Omega): \ \|v\|=1 }\left(
\|  \bar{h}^m_{\mathrm{I}+1} \|_{2,\partial\Omega}\| v \|_{2,\partial\Omega} + \right.\\
+\left(\max (a^\#_{ij}) \|\nabla \mathbf{u}^{m}\|_{2,\Omega} + \max (F^\#_j)
\| \nabla\phi^m \|_{2,\Omega} \right)\|\nabla  v \|_{2,\Omega} +\nonumber \\ \left.
+\| (\gamma^\# |u^{m}|^{\ell-2}+\gamma_1)u^{m}\|_{\ell ',\partial \Omega}\| v \|_{\ell,\partial \Omega} \right)   .
\end{eqnarray*}
Applying Proposition \ref{propu}, we find
\begin{eqnarray*}
{\tau}^{-1}   \int_\tau^T\| \widetilde u_i^M - \widetilde u_i^{M-1}\|_{V'(\Omega)} \mathrm{dt}=\sum_{k=1}^M
 \| \widetilde u_i^k - \widetilde u_i^{k-1}\|_{V'(\Omega)}\leq C; \\
{\tau}^{-1}   \int_\tau^T \| \widetilde u^M - \widetilde  u^{M-1}\|_{V'_\ell(\Omega)}  \mathrm{dt}=\sum_{k=1}^M
  \| \widetilde u^k - \widetilde  u^{k-1}\|_{V'_\ell(\Omega)} \leq C,
\end{eqnarray*}
with $C>0$ being  constants independent on $M$. Taking the Kondrachov-Sobolev embedding
 $H^1(\Omega)\hookrightarrow \hookrightarrow L^2(\Omega)$ and
 $H^1(\Omega)\hookrightarrow \hookrightarrow L^2(\partial\Omega)$,
 we conclude  the proof of strong convergences
  of $\widetilde {\mathbf{u}}^M$ due to the Aubin-Lions theorem \ref{tal}.

 To prove the convergence (\ref{bbm}), we will apply Lemma \ref{lbmm2}.
Considering the weak convergence of $\widetilde u^M$ established in Proposition \ref{propu} and the estimate (\ref{cotaul}),
in order to apply Lemma \ref{lbmm2} it remains to prove that the condition (\ref{cbb}) is fulfilled.
Let $0<z<T$ be arbitrary. Since the objective is to find convergences, it suffices to take $M>T/z$, which means $\tau<z$.
Thus, there exists $k\in\mathbb{N}$ such that $k\tau<z\leq (k+1)\tau$. Moreover, we may choose $M>k+1$ deducing
\begin{eqnarray}
\int_0^{T-z}\int_\Omega (B(\widetilde u^M(t+z))-B(\widetilde u^M(t)))
 (\widetilde u^M(t+z)-\widetilde u^M(t))\mathrm{dx} \mathrm{dt}\leq \nonumber \\
\leq \sum_{l=1 }^{M-k} \int_{(l-1)\tau}^{(l+k)\tau}
\int_\Omega (B(u^{l+k })-B(u^{l })) (u^{l+k }-u^{l })\mathrm{dx} .\label{bbll}
\end{eqnarray}

Let us sum up (\ref{wvfttm}) for $m=l+1,\cdots,l+k$ and multiply by $\tau$, obtaining
\begin{equation}\label{bbl}
\int_\Omega (B(u^{l+k })-B(u^{l })) v \mathrm{dx} \leq \mathcal {I }_{\partial\Omega}^l +\mathcal{I}_\Omega^l ,
\end{equation}
where
\begin{eqnarray*}
\mathcal{I}_{\partial\Omega}^l &:=&  \tau\sum_{m=l+1}^{l+k} \int_{\partial\Omega} | (\gamma(u^m) u^{m} -
 \bar{h}^m_{\mathrm{I}+1}) v| \mathrm{ds} ;\\
\mathcal{I}_\Omega^l &:=& \tau \sum_{m=l+1}^{l+k}\int_\Omega |( 
\sum_{j=1}^{\mathrm{I}+1}a_{\mathrm{I}+1,j} ( \mathbf{u}^m )\nabla u_j^{m} +
  F_{\mathrm{I}+1} ( \mathbf{u}^m) \nabla\phi^m )\cdot \nabla v|  \mathrm{dx}. 
\end{eqnarray*}

Applying the H\"older inequality and using the assumptions (\ref{gamm}), (\ref{aij}) and (\ref{hfj}), we deduce
\begin{eqnarray*}
\mathcal{I}_{\partial\Omega}^l &\leq &  \int_{l\tau}^{(l+k)\tau}\left(
\gamma^\#\|\widetilde u^M\|^{\ell -1}_{\ell,{\partial\Omega} }+ \gamma_1\|\widetilde u^M\|_{\ell ',{\partial\Omega} }+
\| h _{\mathrm{I}+1} \|_{\ell ',{\partial\Omega} }\right)  \|v \|_{\ell,{\partial\Omega}} \mathrm{dt} ; \\
\mathcal{I}_\Omega^l &\leq & \int_{l\tau}^{(l+k)\tau}\left(
\sum_{j=1}^{\mathrm{I}+1}a_{\mathrm{I}+1,j}^\# \|\nabla \widetilde u^M_j \|_{2,\Omega}   +
  F_{\mathrm{I}+1} ^\#\| \nabla\phi^M \|_{2,\Omega}  \right) 
\|\nabla  v\|_{2,\Omega} \mathrm{dt} . 
\end{eqnarray*}
Making use of
 the H\"older inequality and the estimate (\ref{cotaul}) in the above inequalities, we conclude from  (\ref{bbl})
\[ 
\int_\Omega (B(u^{l+k })-B(u^{l })) v \mathrm{dx} \leq 
 \|v \|_{\ell,{\partial\Omega}} C (k\tau)^{1/\ell } +  \|\nabla v\|_{2,\Omega}C\sqrt{k\tau} .
\]
Taking $v=u^{l+k }-u^l$ in the above inequality, firstly gathering with (\ref{bbll}), secondly
applying  the H\"older inequality and after  the estimate (\ref{cotaul}), we obtain
\begin{eqnarray*}
\int_0^{T-z}\int_\Omega (B(\widetilde u^M(t+z))-B(\widetilde u^M(t)))
 (\widetilde u^M(t+z)-\widetilde u^M(t))\mathrm{dx} \mathrm{dt}\leq \nonumber \\
\leq \sum_{l=1 }^{M-k} \int_{(l-1)\tau}^{(l+k)\tau} \left( \|u^{l+k }-u^l \|_{\ell,{\partial\Omega}} C (k\tau)^{1/\ell } + 
 \|\nabla (u^{l+k }-u^l)\|_{2,\Omega}C\sqrt{k\tau}
\right) \leq \\ \leq
 C\left( (k\tau)^{1/\ell }(k\tau+\tau)^{ 1/\ell'}  +  (k\tau)^{1/2} (k\tau+\tau )^{1/2}\right)=C 
 \left( 2^{ 1/\ell'}+ 2^{ 1/2}\right) z.
\end{eqnarray*}
which implies  (\ref{cbb}). 

Thus,
all hypothesis of  Lemma \ref{lbmm2} are fulfilled.
Therefore, Lemma \ref{lbmm2} assures that $B(u^M)$ strongly converges to $B(u)$ in $L^1(Q_T)$,
 which concludes  the proof of  (\ref{bbm}).
   \end{proof}

\begin{proposition}\label{duuw}
If $\mathbf{Z}$ satisfies Proposition \ref{zz}, then
\[
\mathbf{Z}=\partial_t\left(\mathbf{u}, B(u)\right)\mbox{ in }[L^{2}(0,T;(V(\Omega))')]^\mathrm{I}\times L^{\ell'}(0,T;(V_\ell(\Omega))'),
\]
in the weak sense (cf. Remark \ref{rtd}).
\end{proposition}
\begin{proof}
 Let  $t\in ]0,T[$ be arbitrary, but a  fixed number. Thus, there exists
$m\in\{1, \cdots, M\}$ such that  $t\in]t_{m-1,M} , t_{m, M}]$.  
For $j=1,\cdots,\mathrm{I}+1$, we have
\begin{eqnarray*}
\int^t_0 \widetilde Z_j^M(z)\mathrm{dz}=\sum_{k=1}^{m-1}\int_{(k-1)\tau}^{k\tau} Z_j^k \mathrm{dz}
+\int_{(m-1)\tau}^{t} Z_j^m \mathrm{dz}= \\
= \tau\sum_{k=1}^{m-1} Z_j^k
+ (t-(m-1)\tau ) Z_j^m\quad\mbox{ in }\Omega.
\end{eqnarray*}
From the definitions (\ref{defzj})-(\ref{defz1})  we have
\[
\int^t_0 \widetilde{\mathbf{Z}}^M (z)\mathrm{dz}
=\left\{\begin{array}{ll}
U_j^M(t)-u_j^0&\mbox{ for }j=1,\cdots,\mathrm{I}\\
B^M (t) -B(u^0) &\mbox{ for }j=\mathrm{I}+1
\end{array}\right. .
\]
The bounded linear functional $\mathbf{v}\in \mathbf{L}^2(\Omega)\mapsto \int^t_0 
(\widetilde{\mathbf{Z}}^M (z),\mathbf{v})\mathrm{dz}$
is (uniquely) representable   by the element $\left(\mathbf{U}^M-\mathbf{u}^0, B^M-B(u^0)\right)$ 
from $\mathbf{L}^2(\Omega)$ due to the Riesz theorem.

Observing that by the application of the change of variables we have
\[
\int_0^T\int_\Omega u(x,t-\tau)v(x,t)\mathrm{dx}\mathrm{dt}=
\int_{-\tau}^{T-\tau} \int_\Omega u(x,t)v(x,t +\tau )\mathrm{dx}\mathrm{dt},
\]
for every $u,v\in L^2(Q_T)$, 
we find, for $i=1,\cdots,\mathrm{I}$,
\begin{eqnarray*}
\mathcal{J}_i^M:=\int_0^T\int_\Omega \widetilde Z_i^M v \mathrm{dx}\mathrm{dt}=
\frac{1}{\tau}\left(
\int_{T-\tau}^{T} \int_\Omega  u_i^M (x) v(x,t)\mathrm{dx}\mathrm{dt} \right. \\ \left. -
\int_{0}^{T-\tau} \int_\Omega \widetilde u_i^M(x,t) \triangle_\tau v(x,t)\mathrm{dx}\mathrm{dt}-
\int_{-\tau}^{0}\int_\Omega u_i^0(x)v(x,t +\tau )\mathrm{dx}\mathrm{dt}
 \right)   ;\\
\mathcal{J}_{\mathrm{I}+1}^M:=\int_0^T\int_\Omega \widetilde Z_{\mathrm{I}+1}^Mv\mathrm{dx}\mathrm{dt}=
\frac{1}{\tau}\left(
\int_{T-\tau}^{T} \int_\Omega B(u^M) (x) v(x,t)\mathrm{dx}\mathrm{dt}  \right. \\ \left. -
\int_{0}^{T-\tau} \int_\Omega B(\widetilde u^M)(x,t) \triangle_\tau v(x,t)\mathrm{dx}\mathrm{dt} -
\int_{-\tau}^{0}\int_\Omega B(u^0)(x)v(x,t +\tau )\mathrm{dx}\mathrm{dt}  \right) ,
\end{eqnarray*}
 where $\triangle_\tau v(x,t)= v(x,t+\tau)-v(x,t)$  for a.e. $(x,t)\in Q_T$. 

The objective is to pass to the limit $\mathcal{J}_j^M$, for $j=1,\cdots,\mathrm{I}+1$,
 as $M$ tends to infinity. To this end, each term is separately evaluated.

 Firstly, the weak convergence (\ref{zm}) assures that
 \[
\mathcal{ J} ^M\build\longrightarrow^{M\rightarrow\infty} \langle \mathbf{Z},\mathbf{v}\rangle ,
\] 
for all $\mathbf{v}\in [L^{2}(0,T;V(\Omega))]^\mathrm{I}\times L^{\ell}(0,T;V_\ell(\Omega))$.

Considering that $ \| v(T)\|_{2,\Omega}=0$, we evaluate the following term as follows
\[\frac{1}{\tau}\left|
\int_{T-\tau}^{T} \int_\Omega  u_i^M (x) v(x,t)\mathrm{dx}\mathrm{dt}\right|\leq \|u_i^M\|_{2,\Omega}
\frac{1}{\tau}
\int_{T-\tau}^{T} \| v\|_{2,\Omega}\mathrm{dt}\build\longrightarrow^{M\rightarrow\infty} 0,
\]
with Proposition \ref{propu} ensuring the  uniform
 boundedness of $u_i^M$ in $L^2( \Omega)$. 
Considering that $ \| v(T)\|_{\infty,\Omega}=0$ and that Proposition \ref{pbttm} ensures the  uniform
 boundedness of $B(u^M)$ in $L^1( \Omega)$, the similar following term is evaluated as follows
 \[\frac{1}{\tau}\left|
\int_{T-\tau}^{T} \int_\Omega  B(u^M) (x) v(x,t)\mathrm{dx}\mathrm{dt}\right|\leq \|B(u^M)\|_{1,\Omega}
\frac{1}{\tau}
\int_{T-\tau}^{T} \| v\|_{\infty,\Omega}\mathrm{dt}\build\longrightarrow^{\tau\rightarrow 0} 0.
\]

The  difference quotient $\triangle_\tau /\tau$ approximates the time derivative $\partial_t $, that is, 
$\triangle_\tau v /\tau \rightarrow \partial_t v$  a.e. in $ Q_T$ as $\tau$ tends to zero. Moreover, it  verifies
\[
\|\triangle_\tau v\|_{L^1(\tau,T;X)}\leq \|\partial_t  v\|_{L^1(0,T;X)},
\]
with $X$ being a Banach space, whenever $\partial_t v \in L^1(0,T;X)$.
Thanks to Proposition \ref{pbttm}, up to a subsequence 
$\widetilde {\mathbf{u}}^M \rightarrow \mathbf{u}$ and $B(  u^M) \rightarrow B(u) $ a.e. in $ Q_T $. Hence, there hold
\begin{eqnarray*}
\frac{1}{\tau} \int_{0}^{T-\tau} \int_\Omega \widetilde u_i^M(x,t) \triangle_\tau v_i(x,t)\mathrm{dx}\mathrm{dt} &
\build\longrightarrow^{M\rightarrow\infty}& 
\int^T_0 \int_\Omega u_i \partial_t  v_i\mathrm{dx} \mathrm{dt} ; \\
\frac{1}{\tau} \int_{0}^{T-\tau} \int_\Omega B(\widetilde u^M)(x,t) \triangle_\tau v(x,t)\mathrm{dx}\mathrm{dt}&\build\longrightarrow^{M\rightarrow\infty}& \int^T_0\int_\Omega B(u) 
\partial_t  v\mathrm{dx} \mathrm{dt} ,
\end{eqnarray*}
for all $v_i,v\in L^{2}(0,T;H^1(\Omega))$ such that $\partial_t v_i\in L^2(Q_T)$ and $\partial_t v\in L^\infty(Q_T)$.

For $v\in W^{1,1}(0,T;L^2(\Omega))\hookrightarrow C([0,T];L^2(\Omega))$, we have
\[
\frac{1}{\tau}\int_{-\tau}^{0}v(t +\tau )\mathrm{dt}=
\frac{1}{\tau}\int_{0}^{\tau}v(t )\mathrm{dt}\build\longrightarrow^{\tau\rightarrow 0} v(0)\quad
\mbox{ in } L^2(\Omega).
\]

Therefore, we find (\ref{rtdi})-(\ref{rtd1}).
   \end{proof}

Finally, we are in condition in establishing 
the passage to the limit as time goes to zero ($M\rightarrow +\infty$) in   the 
Neumann-Robin elliptic problems (\ref{wvfm})-(\ref{wvfphim}).
\begin{proposition}\label{existence}
Let  $(\mathbf{u},\phi)$ be in accordance with Proposition \ref{propu}, then the pair solves (\ref{wvfi})-(\ref{wvfphi}),
\textit{i.e.} it is the required solution to Theorem \ref{texist}.
\end{proposition}
\begin{proof}
Let $(\widetilde{\mathbf{u}}^M,\phi^M)$ the corresponding Rothe sequence of the steady-state solutions to the 
variational system (\ref{wvfm})-(\ref{wvfphim}). For each $M\in\mathbb{N}$, it satisfies
\begin{eqnarray}
\int^T_0\langle Z^M_i ,v_i \rangle \mathrm{dt}+\sum_{j=1}^{\mathrm{I}+1} 
\int_{Q_T} a_{i,j}( \widetilde{ \mathbf{u}}^M )\nabla\widetilde u_j^M \cdot\nabla v_i\mathrm{dxdt}=\nonumber \\
 = -
\int_{Q_T} F_i(\widetilde{ \mathbf{u}}^M )\nabla \widetilde\phi^M \cdot\nabla v_i\mathrm{dxdt} 
+\int_{\Sigma_T}  h^M_i v_i \mathrm{dsdt}, \quad i= 1,\cdots, \mathrm{I}; \qquad  \label{wvfiM} \\
\int^T_0\langle Z^M, v\rangle \mathrm{dt}+\sum_{j=1}^{\mathrm{I}+1} 
\int_{Q_T} a_{\mathrm{I}+1,j}(\widetilde{ \mathbf{u}}^M )\nabla\widetilde u^M_j \cdot\nabla v\mathrm{dxdt}
+\int_{\Sigma_T} \gamma( \widetilde u^M)\widetilde u^Mv \mathrm{dsdt} =\nonumber \\
 = -
\int_{Q_T} F_{\mathrm{I}+1}(\widetilde{ \mathbf{u}}^M )\nabla \widetilde\phi ^M\cdot\nabla v\mathrm{dxdt} 
+\int_{\Sigma_T}  h^M_{\mathrm{I}+1}v \mathrm{dsdt}; \qquad \label{wvfi1M} \\
\int_{Q_T}\sigma (\widetilde{ \mathbf{u}}^M )\nabla\widetilde \phi^M \cdot\nabla w\mathrm{dx}
 = - \sum_{j=1}^{\mathrm{I}+1} \int_{Q_T} G_j( \widetilde{\mathbf{u}}^M )
\nabla\widetilde u^M_j \cdot \nabla w \mathrm{dx}+\int_0^T\int_\Gamma g w \mathrm{ds} ,\quad  \label{wvfphiM}
\end{eqnarray}
for all $v_i\in L^2(0,T;V(\Omega) )$,  $v\in V_\ell(Q_T)$, and $w\in V(\partial\Omega) $.

 Applying Proposition \ref{pbttm}, and
 the Krasnoselski theorem to the Nemytskii operators   $\mathsf{A}$, $\mathbf{F}$, $\mathbf{G}$,
 and $\sigma$, we have
\begin{eqnarray*}
a_{i,j}( \widetilde{\mathbf{u}}^M )\nabla v \longrightarrow  a_{i,j}( \mathbf{u} )\nabla v &\mbox{ in }& \mathbf{L}^2(Q_T); \\
 F_j ( \widetilde{\mathbf{u}}^M )\nabla v \longrightarrow F_j ( \mathbf{u} )\nabla v &\mbox{ in }& \mathbf{L}^2(Q_T); \\
G_j ( \widetilde{\mathbf{u}}^M )\nabla v \longrightarrow G_j ( \mathbf{u} )\nabla v &\mbox{ in }& \mathbf{L}^2(Q_T); \\
 \sigma ( \widetilde{\mathbf{u}}^M )\nabla v \longrightarrow \sigma ( \mathbf{u} )\nabla v &\mbox{ in }& \mathbf{L}^2(Q_T)
  \quad\mbox{ as }   M\rightarrow +\infty ,
\end{eqnarray*}
 for every $i,j=1,\cdots,\mathrm{I}+1$, and for all $v\in H^1(\Omega)$. 
Thanks to Propositions \ref{propu},  \ref{zz} and \ref{duuw},  we may pass to the limit 
in (\ref{wvfiM}) and (\ref{wvfphiM}), as $M$ tends to infinity, concluding that $u_i$ and $\phi$ verify, respectively,
(\ref{wvfi}), for $i=1,\cdots,\mathrm{I}$, and (\ref{wvfphi}).

Similar argument is valid to pass to the limit in (\ref{wvfi1M}),
 considering that
 \begin{eqnarray}\label{cum}
\nabla (\widetilde{\mathbf{u}}^M ,\widetilde\phi^M)\rightharpoonup \nabla  (\mathbf{u},\phi )
 &\mbox{ in }& \mathbf{L}^2(Q_T);\\ \label{cumg}
\gamma(\widetilde{u}^M ) v\rightarrow \gamma(u)v
 &\mbox{ in }& \mathbf{L}^{\ell/(\ell-1)}(\Gamma_\mathrm{w}\times ]0,T[);\\ \label{cumb}
\widetilde u^M\rightharpoonup u
 &\mbox{ in }& \mathbf{L}^\ell(\Gamma_\mathrm{w}\times ]0,T[),
\end{eqnarray}
and that  $\gamma( \widetilde u^M)v$ strongly converges to $\gamma(u)v$
in $L^2(\Gamma\times ]0,T[)$, which corresponds to the Robin-type boundary condition ($\ell=2$).
   \end{proof}

\section{Application example}
\label{sappl}

 The domain $\Omega$ stands for the representation of electrolysis cells (see Fig. \ref{cell}).
Electrolysis of metals are well known for
lead bromide, magnesium chloride, potassium chloride, sodium chloride, and zinc chloride, to mention a few.

The phenomenological fluxes ${\bf q}$,  ${\bf J}_i$
 and $\bf j$ are, respectively, the measurable heat flux (in W m$^{-2}$), 
 the ionic flux of component $i$ (in mol m$^{-2}$ s$^{-1}$),
  and  the electric current density  (in C m$^{-2}$ s$^{-1}$),
  and they are explicitly driven by gradients of
  the temperature  $\theta$,
 the molar concentration vector ${\bf c}=(c_1,\cdots, c_\mathrm{I})$,
and  the electric potential $\phi$,
in the form (up to some temperature and concentration dependent
factors) \cite{adams,agar,lap2017,wurger,wilson,wu}
\begin{eqnarray}\label{defq}
{\bf q} &=&-k (\theta) \nabla\theta-R\theta^2 \sum_{i=1}^\mathrm{I} D'_i ( c_i, \theta) 
\nabla c_i
-\Pi (\theta) \sigma (\mathbf{c},\theta) \nabla\phi;\\
{\bf J}_i &=&-c_iS_i( c_i ,\theta) \nabla\theta- D_{i} (\theta) \nabla c_i
-u_ic_i\nabla\phi, \quad (i=1,\cdots,\mathrm{I} ); \\
{\bf j} &=&-\alpha_\mathrm{S} (\theta) \sigma (\mathbf{c},\theta) \nabla\theta-F
\sum_{i=1}^\mathrm{I} z_{i}D_{i}(\theta) \nabla c_i
-\sigma(\mathbf{c},\theta) \nabla\phi. \label{defj}
\end{eqnarray}
It includes the Fourier law (with  the thermal conductivity  $k$),
the Fick law (with the diffusion coefficient $D_i$),
 the Ohm law  (with  the electrical conductivity  $\sigma$),
the Peltier-Seebeck cross effect (with the Peltier coefficient
$\Pi$ and the Seebeck coefficient 
 $\alpha_\mathrm{S}$ being correlated by the first Kelvin relation), and the Dufour-Soret  cross effect
 (with the Dufour coefficient $D'_i$ and the Soret coefficient $S_i$).
 Hereafter the subscript $i$ stands for the correspondence to the
ionic component $i$ 
intervened in the reaction process.  Table \ref{tab}  displays the universal constants $R$ and $F$. 
\begin{table}[h]
\centering  \caption{Universal constants} \label{tab}\small
\begin{tabular}{|c|c|c|}
\hline
 $F$  & Faraday constant &$9.6485 \times 10^4 $ C mol$^{-1}$\\
\hline
 $R$ & gas constant & $8.314 $ J mol$^{-1} $K$^{-1}$\\
\hline
$\sigma_{\rm SB}$& Stefan-Boltzmann constant
& $5.67\times 10^{-8}$ W m$^{-2} $K$^{-4}$\\
& (for blackbodies)&\\
\hline
\end{tabular}
\end{table}

Every ionic mobility $u_i=z_{i}D_iF/(R\theta)$ 
 satisfies the Nernst-Einstein relation  $\sigma_i=F{z_i}{}u_ic_i$, with $\sigma_i=t_i\sigma$ representing
ionic conductivity, and $t_i$ is the transference 
number (or transport number) of species $i$.
Indeed, the electrical conductivity is function of the temperature and the concentration vector 
as reported in the Debye and H\"uckel theory \cite{debye}.
After several approximation attempts \cite{her}, the most accepted approximation is the  Debye-H\"uckel-Onsager
equation.
The thermal conductivity of the electrodes can  significantly vary
from sample to sample due to the variability in manufacturing
techniques, carbon paper grades and amounts of particular
compounds. The thermal conductivity is  frequently  estimated to be in
the range 0.1 to 1.6\,W m$^{-1}$ K$^{-1}$, based on the material composition.
 In particular,
the thermal conductivity of nonmetallic liquids under normal conditions is much lower than that of metals and ranges from 0.1 to
 0.6\,W m$^{-1}$ K$^{-1}$, while the thermal conductivity of liquids may change by a factor of 1.1 to 1.6,
 in the interval between the melting point and the boiling point.

Let $T>0$ be an arbitrary (but preassigned) time.
From the conservation of  energy,  the mass balance equations, and
 the conservation of electric charge, we derive, respectively, in $Q_T=\Omega \times ]0,T[$
\begin{align}\label{teq}
\rho c_\mathrm{v}\frac{\partial \theta}{\partial t}+\nabla\cdot{\bf q}=0;\\
\frac{\partial c_i}{\partial t}+\nabla\cdot{\bf J}_i=0;\\
\nabla\cdot{\bf j}=0,\label{peq}
\end{align}
 where the density $\rho$ and the
specific heat capacity $c_\mathrm{v}$ (at constant volume) are assumed  to be dependent  on temperature
and space variable. 
The absence of external forces, assumed in (\ref{teq})-(\ref{peq}),
 is due to their occurrence at the surface of the electrodes $\Gamma_l$ $ (l= \mathrm{a}, \mathrm{c})$, 
 \textit{i.e.},
for a.e. in $]0,T[$, 
\begin{equation}\label{heat}
 {\bf q}\cdot{\bf n}_l=h_\mathrm{C}(\theta-\theta_l) ,\quad
 -Fz_i{\bf J}_i\cdot{\bf n}_l=g_{i,l} ,\quad
 -{\bf j}\cdot{\bf n}=g.
\end{equation}
Here,
$h_{\rm C}$ denotes the conductive heat transfer coefficient, $\theta_l$ denotes a prescribed surface temperature,
$g_{i,l}$ may represent 
a truncated version of the Butler-Volmer expression 
  (cf. \cite{jfpta,lap2017} and the references therein), and 
$g$ denotes a prescribed surface electric current assumed to be
tangent to the surface for all $t>0$. 

The parabolic-elliptic system (\ref{teq})-(\ref{peq}) is accomplished by (\ref{heat}) and the remaining boundary conditions.
For a.e. in $]0,T[$, we consider 
\begin{eqnarray}
 {\bf q}\cdot{\bf n}=h_\mathrm{R} | \theta | ^{\ell-2}\theta -h && \mbox{on }\Gamma_\mathrm{w} ;\label{radia} \\
 \mathbf{q}\cdot{\bf n}=0 && \mbox{on }\Gamma_\mathrm{o};\label{bco}\\
{\bf J}_i\cdot{\bf n}={\bf j}\cdot{\bf n}=0 && 
\mbox{on }\Gamma_\mathrm{w}\cup\Gamma_\mathrm{o},\qquad (i=1,\cdots,I).\label{bcwo}
\end{eqnarray}
The radiative condition (\ref{radia}), with a general exponent $\ell\geq 2$ \cite{lap2017} and
$h_{\rm R}$  denoting the radiative heat transfer coefficient that may 
 depend both on the space variable and the temperature function $\theta$,
accounts, for instance, for the radiation behavior of the heavy water electrolysis,
namely the  Stefan-Boltzmann radiation law if $\ell=5$, {\em i.e.}
$h_\mathrm{R}=\sigma_{\rm SB}\epsilon$,
and $h=\sigma_{\rm SB}\alpha \theta_\mathrm{w}^4$.
The parameters, $\epsilon$ and  $\alpha$, represent  the emissivity and the absorptivity, respectively,
  $\theta_\mathrm{w}$ denotes a prescribed wall surface temperature, and $\sigma_{\rm SB}$ stands for
 Stefan-Boltzmann constant for blackbodies (cf. Table \ref{tab}).

 \begin{definition}
 We call by the thermoelectrochemical (TEC) problem the finding of the temperature-concentration-potential triplet
 $(\theta,\mathbf{c},\phi)$ satisfying  (\ref{teq})-(\ref{peq}), under (\ref{defq})-(\ref{defj}), accomplished with 
(\ref{heat})-(\ref{bcwo}), and the initial conditions $\theta(0)=\theta_0$ and $ \mathbf{c} (0)= \mathbf{c}^0 $ in $\Omega$.
 \end{definition}
 
We assume 
\begin{description}
\item[(A1)]  The coefficients $\rho$ and $c_\mathrm{v}$  are assumed to be Carath\'eodory functions 
from $\Omega\times\mathbb{R}$ into $\mathbb{R}$. Moreover,
there exist  $b_\#,b^\#>0$ such that
\[
b_\#\leq \rho(x,e)c_\mathrm{v}(x,e)\leq b^\#,
\] 
for a.e. $x\in\Omega$, and for all $e\in\mathbb{R}$. 
Although  the specific heat coefficient of most liquid metals for which data are available is negative, 
it is positive at high temperatures, and often invariant with temperature.

\item[(A2)] 
 The electrical and thermal  conductivities, Peltier, Seebeck, Soret, Dufour,
and diffusion coefficients $\sigma,k,\Pi,\alpha,S_i,D_i',D_i$ ($i=1,\cdots,\mathrm{I}$) are
  Carath\'eodory functions such that verify (\ref{smm}),
   \begin{eqnarray*}
\exists k_\#, k^\#>0:\quad& k_\#\leq k(x,e)\leq k^\#;\\
\exists\Pi^\#>0:\qquad&
|\Pi(x,e)|\leq \Pi^\# ;\\
\exists\alpha^\#>0:\qquad&|\alpha_\mathrm{S}(x,e)|\leq \alpha^\#; \\
\exists S_i^\#>0:\qquad& |dS_i(x,d,e)|\leq S_i^\#; \\
\exists (D_i')^\#>0:\qquad&  Re^2|D_i'(x,d,e)|\leq (D_i')^\#; \\
\exists D_i^\#>0:\qquad& F|z_i|D_i(x,e)\leq D_i^\#; \\
\exists (D_i)_\#>0:\qquad& D_i(x,e)\geq (D_i)_\#, 
  \end{eqnarray*}
for  a.e. $ x\in \Omega$,    and for all $d,e\in \mathbb R$.

\item[(A3)]  The  transference  coefficient $t_i\in L^\infty (\Omega)$
is such that
   \[
\exists t_i^\#>0:\quad
0\leq t_i(x)\leq F|z_i|t_i^\#, \quad\mbox{for a.e. } x\in \Omega.\]

 \item[(A4)]  The boundary operator $h_\mathrm{R}$ is a
 Carath\'eodory  function from $\Gamma_{\mathrm{w}}\times\mathbb{R}$
 into $\mathbb{R}$ such that verifies
 \[
\exists \gamma_\#, \gamma^\#>0:\quad \gamma_\#\leq h_\mathrm{R}(x,e)\leq  \gamma^\#
\quad\mbox{for a.e. } x\in \Gamma_{\mathrm{w}},\quad\forall e\in \mathbb R
.\]

\item[(A5)]  The boundary function
 $h_\mathrm{C}$ is  measurable  from  $\Gamma\times ]0,T[$ into $\mathbb{R}$
  satisfying
 \[
\exists h_\#,h^\#>0:
\qquad h_\#\leq h_\mathrm{C}(x)\leq h^\#
,\quad \mbox{for a.e. } x\in \Gamma.
  \]

\item[(A6)] $g\in L^{2}(\Gamma)$,
  $\theta_{\mathrm{a}}\in L^{2}(\Gamma_{\mathrm{a}}\times ]0,T[)$,
  $\theta_{\mathrm{c}}\in L^{2}(\Gamma_{\mathrm{c}}\times ]0,T[)$ and
 $h\in L^{\ell/(\ell-1)}(\Gamma_{\mathrm{w}}\times ]0,T[)$.

\item[(A7)]  For 
 each $i=1,\cdots, \mathrm{I}$, $g_{i,\mathrm{a}}$ and $g_{i,\mathrm{c}}$ belong to 
$ L^{2}(\Gamma_\mathrm{a}\times ]0,T[)$ and  $ L^{2}(\Gamma_\mathrm{c}\times ]0,T[)$, respectively.
\item[(A8)]   
$\theta_0,c_i^0\in L^{2}(\Omega)$, $i=1,\cdots,\mathrm{I}$.
\end{description}

The main result of existence to the TEC problem is the following theorem.
\begin{theorem}\label{texist2}
Let the assumptions (A1)-(A8) be fulfilled. In addition,
suppose that the smallness conditions
 \begin{eqnarray}\label{ss1}
 (D_i)_\# &>& \frac{1}{2}\left(S^\#_i+ (D'_i)^\# + t_i^\#\sigma^\#+ D_i^\#\right),  \quad i=1,\cdots,
\mathrm{I};\\  \label{ss2}
 k_\# &>& \frac{1}{2}\left(\sum _{ j=1}^\mathrm{I} \left[
S^\#_j + (D'_j )^\#\right] + \Pi^\#\sigma^\# +\alpha^\#\sigma^\# \right);\\
 \sigma_\# &>& \frac{1}{2} \left(\sum _{j=1}^{\mathrm{I}}( t_j^\#\sigma^\#+ D_j^\#)
+(\Pi^\# +\alpha^\#)\sigma^\#\right) 
 \end{eqnarray}
  hold. Then,
there exists  at least one  weak solution to the TEC problem in the following sense
\begin{eqnarray*}
\int^T_0\langle \partial_t c_i ,v_i \rangle \mathrm{dt}+
\int_{Q_T} D_{i}(c_i,\theta)\nabla c_i \cdot\nabla v_i\mathrm{dxdt}= \sum_{l=\mathrm{a},\mathrm{c}}
\int^T_0\int_{\Gamma_l}  g_{i,l} v_i \mathrm{dsdt} \\
  -
\int_{Q_T} \left( c_iS_i( c_i ,\theta) \nabla\theta + 
t_i (Fz_i)^{-1}\sigma(\mathbf{c},\theta) \nabla\phi\right) \cdot\nabla v_i\mathrm{dxdt} 
, \quad i= 1,\cdots, \mathrm{I}; \\
\int^T_0\langle\rho(\theta) c_\mathrm{v}(\theta)\partial_t  \theta , v\rangle \mathrm{dt}+
\int_{Q_T} k( \theta)\nabla\theta \cdot\nabla v\mathrm{dxdt}
+\int^T_0\int_\Gamma h_\mathrm{C} \theta v \mathrm{dsdt} +\nonumber \\
+\int_0^T\int_{\Gamma_\mathrm{w}}  h_\mathrm{R}(\theta)|\theta|^{\ell-2} \theta v \mathrm{dsdt} 
 = \sum_{l=\mathrm{a},\mathrm{c}}
\int^T_0\int_{\Gamma_l} h_\mathrm{C}\theta_l v \mathrm{dsdt}
+\int_0^T\int_{\Gamma_\mathrm{w}}  hv \mathrm{dsdt}\\ -
\int_{Q_T} \left( R\theta^2 \sum_{j=1}^\mathrm{I} D'_j ( c_j, \theta) \nabla c_j
+\Pi (\theta) \sigma (\mathbf{c},\theta) \nabla\phi\right) \cdot\nabla v\mathrm{dxdt} ;  \\
\int_\Omega\sigma ( \mathbf{c},\theta )\nabla\phi\cdot\nabla w\mathrm{dx}
 =\int_\Gamma g w \mathrm{ds} \\
 -  \int_\Omega \left( \alpha_\mathrm{S} (\theta) \sigma (\mathbf{c},\theta) \nabla\theta + F
\sum_{j=1}^\mathrm{I} z_{j}D_{j}( c_j,\theta) \nabla c_j\right)
\cdot \nabla w \mathrm{dx}
 , \ \mbox{a.e. in }  ]0,T[, 
\end{eqnarray*}
for all $v_i\in L^2(0,T;V(\Omega) )$,  $v\in V_\ell(Q_T)$, and $w\in V(\partial\Omega) $
where the time derivative is understood in accordance to Remark \ref{rtd}.
\end{theorem} 
\begin{proof}
The existence of weak solutions to the TEC problem is a consequence of Theorem \ref{texist},
under $u_i=c_i$, $i=1,\cdots,\mathrm{I}$ and $u_{\mathrm{I}+1}=\theta$.
The explicit forms of the transport coefficients are $b=\rho c_\mathrm{v}$,
\begin{eqnarray*}
a_{i,j}(\mathbf{c} , \theta)&=&\left\{\begin{array}{ll} 
D_i( c_i ,\theta) \delta_{i,j} &\mbox{if } 1\leq i, j\leq \mathrm{I}\\
c_i S_i(c_i,\theta) &\mbox{if } 1\leq i\leq \mathrm{I},\quad  j=\mathrm{I}+1 \\
R\theta^2 D'_j(c_j,\theta) &\mbox{if } i=\mathrm{I}+1,\quad 1\leq j\leq \mathrm{I}\\
k(\theta) &\mbox{if } i=\mathrm{I}+1,\quad j= \mathrm{I}+1
\end{array}\right. \\
F_{j}(\mathbf{c},\theta)&=&\left\{\begin{array}{ll} 
t_j (Fz_j)^{-1}\sigma(\mathbf{c},\theta) &\mbox{if } 1\leq j\leq \mathrm{I}\\
\Pi(\theta) \sigma(\mathbf{c},\theta) &\mbox{if } j=\mathrm{I}+1 
\end{array}\right. \\
G_{j}(\mathbf{c},\theta)&=&\left\{\begin{array}{ll} 
Fz_j D_j( c_j, \theta) &\mbox{if } 1\leq j\leq \mathrm{I}\\
\alpha_\mathrm{S}(\theta) \sigma(\mathbf{c},\theta) &\mbox{if } j=\mathrm{I}+1 .
\end{array}\right. 
\end{eqnarray*}
The assumption (A1) is exactly (H2).
The assumptions (A2)-(A3) imply (H1) with
\begin{eqnarray*}
F_{j}^\#&=&\left\{\begin{array}{ll} 
t_j ^\#\sigma^\# &\mbox{if } 1\leq j\leq \mathrm{I}\\
\Pi^\# \sigma^\# &\mbox{if } j=\mathrm{I}+1 
\end{array}\right. \\
G_{j}^\# &=&\left\{\begin{array}{ll} 
D_j^\# &\mbox{if } 1\leq j\leq \mathrm{I}\\
\alpha^\#\sigma^\# &\mbox{if } j=\mathrm{I}+1 .
\end{array}\right. 
\end{eqnarray*}
The assumption (A2) implies (H4) and (H3) with
\begin{eqnarray*}
 (a_{i})_\#&=&\left\{\begin{array}{ll} 
(D_i)_\#&\mbox{if } 1\leq i\leq \mathrm{I}\\
k_\#&\mbox{if } i=\mathrm{I}+1
\end{array}\right. \\
a_{i,j}^\#&=&\left\{\begin{array}{ll} 
D_i^\#/(F|z_i|) \delta_{i,j} &\mbox{if } 1\leq i, j\leq \mathrm{I}\\
S_i^\# &\mbox{if } 1\leq i\leq \mathrm{I},\quad  j=\mathrm{I}+1 \\
( D'_j)^\# &\mbox{if } i=\mathrm{I}+1,\quad 1\leq j\leq \mathrm{I}\\
k^\#&\mbox{if } i=\mathrm{I}+1,\quad j= \mathrm{I}+1 .
\end{array}\right. \end{eqnarray*}

Moreover, considering in Section \ref{sfptgalk}
 \begin{eqnarray*}
(L_i)_\# &=& (D_i)_\#- \frac{1}{2}\left(S^\#_i+ (D'_i)^\# + F_i^\#+ G_i^\#\right), \quad i=1,\cdots,
\mathrm{I};\\
(L_{\mathrm{I}+1})_\# &=& k_\#- \frac{1}{2}\left(\sum _{ j=1}^\mathrm{I} \left[
S^\#_j + (D'_j )^\#\right] + F_{\mathrm{I}+1}^\#+ G_{\mathrm{I}+1}^\#\right);\\
(L_{\mathrm{I}+2})_\# &=&
 \sigma_\#- \frac{1}{2}\sum _{ j=1}^{\mathrm{I}+1} \left( F_j^\#+ G_j^\#\right),
 \end{eqnarray*}
the smallness conditions (\ref{saii})-(\ref{small2}) read (\ref{ss1})-(\ref{ss2}). 

Finally, the assumptions (A4)-(A5) fulfill (H5) with
\[
\gamma(x,e)=\left\{\begin{array}{ll}
h_\mathrm{C} (x)&\mbox{if }x\in \Gamma \\
h_\mathrm{R}(x,e)|e|^{\ell-2} &\mbox{if }x\in \Gamma_\mathrm{w} \\
0 &\mbox{otherwise}
\end{array}\right.
\]
for all $e\in\mathbb{R}$,
and (A5)-(A8) fulfill the remaining hypothesis of Theorem \ref{texist}. 
  \end{proof}

\section*{Appendix}

$
\begin{array}{lll}
\textbf{Nomenclature list:}&&\\
c&\mbox{molar concentration (molarity)}&\mbox{mol  m}^{-3}\\
 D &\mbox{diffusion coefficient} &\mbox{m$^2 $s$^{-1}$}\\
 D' &\mbox{Dufour coefficient}&\mbox{m$^2 $s$^{-1} $K}^{-1}\\
 h &\mbox{heat transfer coefficient}&\mbox{W m$^{-2} $K}^{-1}\\
 k &\mbox{thermal conductivity}&\mbox{W m$^{-1}$K}^{-1}\\
 S &\mbox{Soret coefficient (thermal diffusion)}&\mbox{m$^2$s$^{-1} $K}^{-1}\\
 t& \mbox{transference number} &\mbox{(dimensionless)}\\
 u &\mbox{ionic mobility}&\mbox{m$^2 $V$^{-1} $s}^{-1}\\
 z &\mbox{valence} &\mbox{(dimensionless)}\\
\alpha_\mathrm{S}  &\mbox{Seebeck coefficient} &\mbox{V K}^{-1}\\
\phi &\mbox{electric potential}& \mbox{V}\\
 \pi &\mbox{Peltier coefficient} &\mbox{V}\\
\sigma &\mbox{electrical conductivity}&\mbox{S m}^{-1}\\
\theta &\mbox{absolute temperature} &\mbox{K}
\end{array}
$


\begin{thebibliography}{}

\bibitem{adams}
E.M. Adams, I.R. McDonald, K. Singer, 
Collective dynamical properties of molten salts: Molecular dynamics calculations on sodium chloride,
\textit{Proc. R. Soc. Lond. Ser. A Math. Phys. Eng. Sci.} \textbf{357} :1688 (1977), 37-57.

\bibitem{agar}
J.N. Agar, J.C.R. Turner,
Thermal diffusion in solutions of electrolytes,
\textit{Proc. R. Soc. Lond. Ser. A Math. Phys. Eng. Sci.} \textbf{255} :1282 (1960), 307-330.

\bibitem{alp}
A. Alphonse, C.M. Elliott, B. Stinner,
An abstract framework for parabolic PDEs on evolving spaces. 
{\em Port. Math. (N.S.)} {\bf 72} :1 (2015), 1-46.

\bibitem{alt}
H.W. Alt, S. Luckhaus,
Quasilinear elliptic-parabolic differential equations,
\textit{Math. Z.} \textbf{183}  (1983), 311-341.

\bibitem{cheung}
 C.-Y. Cheung, C. Menictas, J. Bao,  M. Skyllas-Kazacos, B.J.  Welch, 
 Spatial thermal condition in aluminum reduction cells under influences of electrolyte flow,
\textit{Chemical Engineering Research and Design} \textbf{100} (2015), 1-14. 


\bibitem{epjp}
 L. Consiglieri, 
On the posedness of thermoelectrochemical coupled systems.
{\em Eur. Phys. J. Plus} {\bf  128} :5  (2013),  47 - 17 pages.

\bibitem{jfpta}
 L. Consiglieri, Sufficient conditions to
the existence for solutions of a thermoelectrochemical problem.
\textit{J. Fixed Point Theory  Appl.} \textbf{17} :4 (2015), 669-692.

\bibitem{lap2017}
L. Consiglieri,
{\em Quantitative estimates on boundary value problems.
Smallness conditions to thermoelectric and thermoelectrochemical problems},
 Lambert Academic Publishing, Saarbr\"ucken 2017.

\bibitem{bumi}
 L. Consiglieri, Weak solutions for  a thermoelectric problem with power-type boundary effects. 
{\em  Boll. Unione Mat. Ital.} {\bf } (2018),
https://doi.org/10.1007/s40574-018-0159-z.



\bibitem{debye}
Von P. Debye, E. H\"uckel,
Zur theorie der elektrolyte. I.
Gefrierpunktserniedrigung und verwandte erscheinungen
(The theory of electrolytes. I. Lowering of freezing point and related phenomena),
 \textit{Physikalische Zeitschrift} \textbf{24} :9 (1923), 185-206.


\bibitem{doudu}
J. Douglas, Jr., T. Dupont,
Galerkin methods for parabolic equations,
\textit{SIAM J. Numer. Anal.} \textbf{7} :4 (1970), 575-626.

\bibitem{dreher}
M. Dreher, A. J\"ungel, Compact families of piecewise constant functions in $L^p(0, T; B)$.
 {\em  Nonlinear Anal.} {\bf 75} (2012), 3072-3077.
 
\bibitem{azab}
M.S. El-Azab, A.A. Ashour,
 Rothe's method to nonlinear parabolic problems with a nonlinear boundary condition,
 \textit{Int. J. Differ. Equ. Appl.} \textbf{9} :3 (2004), 193-212.
 
\bibitem{wurger}
K.A. Eslahian,  A. Majee, M. Maskos, A. W\"urger,
Specific salt effects on thermophoresis of charged colloids, 
{\em Soft Matter} {\bf  10}  (2014), 1931-1936.

\bibitem{fuller}
T.F. Fuller, M. Doyle, J. Newman,
Simulation and optimization of the dual Lithium ion insertion cell,
\textit{J. Electrochem. Soc.} \textbf{141}:1 (1994), 1-10.


\bibitem{gudi}
T. Gudi, A.K. Pani,
Discontinuous Galerkin methods for quasi-linear elliptic problems of nonmonotone type,
\textit{SIAM J. Numer. Anal.} \textbf{45} :1 (2007), 163-192.

\bibitem{her}
H.S. Harned, B.B. Owen,
\textit{The physical chemistry of electrolytic solutions},
Reinhold Publishing Corporation, New York 1943.

\bibitem{nature}
S.W. Hasan, S.M. Said, M.F.M. Sabri, A.S.A. Bakar, N.A. Hashim, M.M.I.M. Hasnan, J.M. Pringle,
D.R. MacFarlane,
High thermal gradient in thermo-electrochemical cells by insertion
of a poly(vinylidene fluoride) membrane, 
\textit{Sci. Rep.} \textbf{6} :29328 (2016), 11 pages. doi: 10.1038/srep29328.

\bibitem{kacur}
 J. Ka\v cur,
Application of Rothe's method to nonlinear evolution equations,
\textit{Mat. \v Cas.} \textbf{25} :1 (1975),  63-81.

\bibitem{kacur99}
 J. Ka\v cur,
 Solution to strongly nonlinear parabolic problems by a linear approximation scheme,
 \textit{IMA J. Numer. Anal.} \textbf{19}  (1999), 119-145.
 
\bibitem{kacurmah}
 J. Ka\v cur, M.S. Makmood,
 Galerkin characteristics method for convection-diffusion problems with memory terms,
 \textit{Int. J. Numer. Anal. Model.} \textbf{6} :1 (2009), 89-109.

\bibitem{kupper}
C. Kupper, W.G. Bessler,
Multi-scale thermo-electrochemical modeling of performance and aging of a LiFePO$_4$/Graphite Lithium-ion cell,
\textit{J. Electrochem. Soc.} \textbf{164} :2 (2017), A304-A320.

\bibitem{juha}
T. Kuusi, L Monsaingeon,  J.H. Videman, Systems of partial differential equations in porous medium,
\textit{Nonlinear Anal.}  \textbf{133} (2016), 79-101.

\bibitem{ll}
J. Leray, J.L. Lions,
 Quelques r\'esultats de Vi\v sik sur les probl\`emes elliptiques non lin\'eaires par les m\'ethodes de
Minty-Browder. {\em Bull. Soc. Math. France} {\bf 93} (1965), 97-107.

\bibitem{pouso}
\'O. L\'opez-Pouso, R. Mu\~noz-Sola,
About the solution of the even parity formulation of the transient radiative heat transfer equations,
\textit{RACSAM  Rev. R. Acad. Cienc. Exactas Fis. Nat. Ser. A. Mat.} \textbf{104} :1 (2010), 129-152.

\bibitem{meth}
R.N. Methekar, P.W.C. Northrop, K. Chen, R.D. Braatz,  V.R. Subramanian, 
Kinetic Monte Carlo simulation of surface heterogeneity in Graphite anodes for Lithium-ion batteries: Passive layer
formation,
\textit{J. Electrochem. Soc.} \textbf{158} :4 (2011), A363-A370.

\bibitem{north}
P.W.C. Northrop, V. Ramadesigan, S. De, V.R. Subramanian, 
Coordinate transformation, orthogonal collocation and model reformulation for simulating electrochemical-thermal behavior of
Lithium-ion battery stacks,
\textit{J. Electrochem. Soc.} \textbf{158} :12  (2011), A1461-A1477.

\bibitem{plus}
V. Pluschke,
Rothe's method for parabolic problems with nonlinear degenerating coefficient,
\textit{Martin-Luther-University Halle, Dept. of Math.} \textbf{Report No. 14} (1996), 17 pages.

 \bibitem{roub}
T. Roub\'\i\v cek,
\textit{Nonlinear Partial differential equations with applications},
Birkh\"auser Verlag, 2005.

\bibitem{sun2005}
S. Sun, M.F. Wheeler.
Discontinuous Galerkin methods for coupled flow and reactive transport problems.
\textit{Appl. Numer. Math.} \textbf{52} (2005), 273-298.

\bibitem{wilson}
J.R. Wilson,
 The structure of liquid metals and alloys,
\textit{Metallurgical Reviews} \textbf{10} :40 (1965), 381-590.

\bibitem{wu}
J. Wu, J. J. Black, L. Aldous,
Thermoelectrochemistry using conventional and novel gelled electrolytes in heat-to-current thermocells, 
\textit{Electrochemical Acta} \textbf{225} (2017), 482-492.

\bibitem{wu-xu}
J. Wu, J. Xu, H. Zou,
On the well-posedness of a mathematical model for Lithium-ion battery systems,
\textit{Methods Appl. Anal.} \textbf{13} :3 (2006), 275-298.

\end{thebibliography}
\end{document}